\documentclass[10pt]{article}
\usepackage{lmodern}
\usepackage{amsmath}
\usepackage[T1]{fontenc}
\usepackage[utf8]{inputenc}
\usepackage{authblk}
\usepackage{amsfonts}
\usepackage{graphicx}
\usepackage{rotating}
\usepackage{amssymb}
\usepackage[english]{babel}
\usepackage{xcolor}
\usepackage{amsthm}
\usepackage{graphicx}
\usepackage{mathrsfs}
\usepackage{makecell}
\usepackage{microtype}
\usepackage{mathscinet}
\usepackage{array}
\usepackage{multirow}
\usepackage{enumerate}
\usepackage[cal=boondoxo,bb=ams]{mathalfa}
\usepackage{hyperref}
\hypersetup{hidelinks}
\usepackage{booktabs}

\newtheorem{theorem}{Theorem}[section]
\newtheorem{prop}{Proposition}[section]
\newtheorem{lemma}{Lemma}[section]
\newtheorem{coro}{Corollary}[section]
\newtheorem{remark}{Remark}[section]

\newcommand{\ml}{\mathcal}
\newcommand{\mb}{\mathbb}
\DeclareMathOperator{\intt}{int}
\DeclareMathOperator{\extt}{ext}
\DeclareMathOperator{\bdd}{bdd}
\DeclareMathOperator{\lin}{lin}
\DeclareMathOperator{\non}{nlin}

\def\XXint#1#2#3{{\setbox0=\hbox{$#1{#2#3}{\int}$ }
		\vcenter{\hbox{$#2#3$ }}\kern-.6\wd0}}

\title{Large time asymptotic behavior for the weakly damped Jordan-Moore-Gibson-Thompson equation}
\author[1]{Wenhui Chen\thanks{Wenhui Chen (wenhui.chen.math@gmail.com)}}
\author[2]{Yan Liu\thanks{Yan Liu (ly801221@163.com)}}
\author[1]{Manqing Luo\thanks{Manqing Luo (1985388097@qq.com)}}

\affil[1]{School of Mathematics and Information Science, Guangzhou University,\authorcr Guangzhou, P.R. China}
\affil[2]{Department of Applied Mathematics, Guangdong University of Finance,\authorcr 510521 Guangzhou, China}

\date{}

\begin{document}
		\maketitle

		\begin{abstract}
			\medskip
This manuscript considers the Jordan-Moore-Gibson-Thompson (JMGT) equation and its linearized equation with an additional weak damping term (proposed by [B. Kaltenbacher, \emph{Inverse Problems} (2025)] firstly) in the whole space $\mathbb{R}^n$. We mainly study the unique existence and large time behavior, including optimal decay estimates and asymptotic profiles, of global in-time Sobolev solutions for any $n\geqslant 1$. This weak damping term leads to diffusion profiles in the sub-critical case $\delta>0$ and  regularity-loss decay properties in the critical case $\delta=0$, which are greatly different from the results for the corresponding classical models without the weak damping term.
			\\
			
\noindent\textbf{Keywords:} Jordan-Moore-Gibson-Thompson equation, asymptotic profile, global in-time solution, regularity-loss decay property, weak damping term\\
			
\noindent\textbf{AMS Classification (2020)} 35B40, 35L75, 35L30, 35A01 
		\end{abstract}
\fontsize{12}{15}
\selectfont

\section{Introduction}\setcounter{equation}{0}\label{Section-Introduction}
\hspace{5mm}In the present paper, we mainly consider the following Cauchy problem for the Kuznetsov-type Jordan-Moore-Gibson-Thompson (JMGT) equation with a weak damping term (proposed by \cite[Equation (1)]{Kaltenbacher-2025} originally) for any $n\geqslant 1$:
\begin{align}\label{Eq-Dissipation-JMGT}
\begin{cases}
\tau\psi_{ttt}+\psi_{tt}-\Delta\psi-(\delta+\tau)\Delta\psi_t+\gamma\psi_t=\partial_t\ml{N}(\psi_t,\nabla\psi),&x\in\mb{R}^n,\ t>0,\\
\psi(0,x)=\psi_0(x),\ \psi_t(0,x)=\psi_1(x),\ \psi_{tt}(0,x)=\psi_2(x),&x\in\mb{R}^n,
\end{cases}
\end{align}
with the thermal relaxation $\tau>0$ from the Cattaneo law of heat conduction, the diffusivity of sound $\delta\geqslant0$, and the weak attenuation coefficient $\gamma>0$, where the unknown function $\psi=\psi(t,x)\in\mb{R}$ is referred to the acoustic velocity potential. The derivative-type nonlinearity is expressed by
\begin{align}\label{Nonlinearity-Intro}
\ml{N}(\psi_t,\nabla\psi):=\frac{B}{2A}|\psi_t|^2+|\nabla\psi|^2
\end{align}
with the constant radio $\frac{B}{2A}>0$ standing for the nonlinearities of state in a given fluid. Our main purpose is to investigate some new influence from the additional weak damping term $+\gamma \psi_t$ on the large time qualitative properties of global in-time solutions to the Cauchy problems for the JMGT equation \eqref{Eq-Dissipation-JMGT} in the sub-critical case $\delta>0$, and its linearized equation \eqref{Eq-Dissipation-MGT} in the sub-critical case $\delta>0$ and the critical case $\delta=0$, separately. 

\subsection{Backgrounds for the classical JMGT equation}\hspace{5mm}
Let us first recall the physical and historical backgrounds for the classical JMGT equation, i.e. the vanishing weak damping limit of \eqref{Eq-Dissipation-JMGT} with $\gamma=0$, as follows:
\begin{align}\label{Eq-Real-JGMT-W}
	\tau\psi_{ttt}+\psi_{tt}-\Delta\psi-(\delta+\tau)\Delta\psi_t=\partial_t\ml{N}(\psi_t,\nabla\psi)
\end{align}
equipping the nonlinearity \eqref{Nonlinearity-Intro}, $\tau>0$, and $\delta\geqslant0$.
To describe the propagation of sound in viscous thermally relaxing fluids, an approximated model of fully compressible Navier-Stokes-Cattaneo (NSC) system in irrotational flows is always considered in nonlinear acoustics. Precisely, by using the Lighthill scheme of approximations (cf. \cite{Lighthill=1956}) for the fully compressible NSC system to retain the terms of first- and second-orders with respect to small perturbations around the constant equilibrium state, the classical JMGT equation \eqref{Eq-Real-JGMT-W} arises. The classical JMGT equation is named after the early works of F.K. Moore, W.E. Gibson \cite{MooreGibson1960} in 1960, of P.A. Thompson \cite{Thompson1972} in 1972, and  of P.M. Jordan \cite{Jordan-2014} in 2014 who employed the Cattaneo law to eliminate an infinite signal speed paradox from the Fourier law of heat conduction in thermoviscous fluids.  These mathematical models have been  used extensively in medical and industrial applications of high-intensity focused
ultrasound, for instance, the medical imaging and therapy, ultrasound cleaning and welding (cf. \cite{Abramov-1999,Dreyer-Krauss-Bauer-Ried-2000,Kaltenbacher-Landes-Hoffelner-Simkovics-2002}).

We next are going to address a brief review on the classical JMGT equation \eqref{Eq-Real-JGMT-W} in the whole space $\mb{R}^n$ (for the bounded domains case, we refer the interested reader to \cite{Kaltenbacher-Lasiecka-Marchand-2011,Marchand-McDevitt-Triggiani-2012,Kaltenbacher-Lasiecka-Pos-2012,Conejero-Lizama-Rodenas-2015,Dell-Pata=2017,Kaltenbacher-Nikolic-2019,B-L-2020,Kaltenbacher-Niko-2021,Niko-Winker=2024} and references given therein).  As preparations, some qualitative properties of solutions for the corresponding linearized Cauchy problem, i.e. the classical Moore-Gibson-Thompson (MGT) equation with the vanishing right-hand side (see \eqref{Eq-Dissipation-MGT} with $\gamma=0$), are well-established in recent years, including sharp energy decay rates in \cite{Pellicer-Said-Houari=2019}, optimal growth/decay $L^2$ estimates in  \cite{Chen-Ikehata=2021,Chen-Takeda=2023}, $L^p-L^q$ estimates with $1\leqslant p\leqslant q\leqslant +\infty$ in \cite{Chen-Gong=2024,Chen-Ma-Qin=2025}, singular limits even with a singular layer in \cite{Chen-Ikehata=2021,Chen-Gong=2024}. Let us turn to the Cauchy problem for \eqref{Eq-Real-JGMT-W}. There are some well-established results for the sub-critical case $\delta>0$. The global in-time well-posedness results of small data Sobolev solutions and Besov solutions, respectively, were proved by \cite{Racke-Said-2020} and \cite{Said-Houari=Besov=2022} in $\mb{R}^3$ via energy methods. Then, the author of \cite{Said-Houari=Large-Norm=2022} removed the smallness assumption for higher-order Sobolev data in $\mb{R}^3$ and deduced energy decay estimates with the additional $L^1$ integrability of initial data. Carrying Sobolev regular small data with the additional $L^1$ integrability, the authors of \cite{Chen-Takeda=2023} derived large time optimal growth/decay estimates and optimal leading terms of global in-time Sobolev solutions in $\mb{R}^n$ with any $n\geqslant 1$ via the WKB analysis and the Fourier analysis, where the diffusion-wave profile leads to the polynomial growth if $n=1$ and the logarithmic growth if $n=2$ for the solution itself. By using the frequency-uniform decomposition techniques, the recent work \cite{Chen=2025} obtained the global in-time existence result without requiring the smallness of rough initial data (but their Fourier support restrictions belong to a suitable subset of first octant) for the classical JMGT equation of Westervelt-type in the complex-valued framework. Note that the global in-time existence result equipping the additional $L^m$ integrable initial data was demonstrated as $m\in[1,2)$.
 Focusing on the critical case $\delta=0$, the blow-up result associated with upper bound estimates for the lifespan in lower dimensions $n\leqslant 3$ was demonstrated by \cite{Chen-Liu-Palmieri-Qin=2023}, whose proof is motivated by the earlier manuscripts  \cite{Chen-Palmieri=2020} for the power-type nonlinearity and \cite{Chen-Palmieri=2021} for the derivative-type nonlinearity in the semilinear MGT equations. Regrettably, the existence of global in-time solution for higher-dimensions $n\geqslant 4$ in the critical case $\delta=0$ remains an open question.

\subsection{Main purposes of this manuscript}
\hspace{5mm}To the best of our knowledge, the weakly damped JMGT equation \eqref{Eq-Dissipation-JMGT} and its linearization \eqref{Eq-Dissipation-MGT} have barely been studied so far (even in bounded domains) with the exception of \cite{Kaltenbacher-2025} from the viewpoint of inverse problems. The purposes of present manuscript are twofold, being to understand how the weak damping term affects the qualitative properties of solution.

Our first contribution is the large time asymptotic properties for the linearized weakly damped MGT equation \eqref{Eq-Dissipation-MGT}. The optimal large time estimates associated with asymptotic profiles are derived for any $\delta\geqslant 0$, where we required the suitably higher $\ell$-order Sobolev regularity for all initial data if $\delta=0$ due to the regularity-loss phenomenon (this regularity-loss decay property is well-understood in many physical models, for example, the Euler-Maxwell system \cite{Hosono-Kawashima=2006}, the dissipative Timoshenko system \cite{Ide-Haramoto-Kawashima=2008}, and the Vlasov-Maxwell system \cite{Duan=2011}). This optimal decay rate is caused by the weak damping term $+\gamma\varphi_t$, and the regularity-loss decay property in the critical case $\delta=0$ is caused by the lack of viscous damping term $-\delta\Delta\varphi_t$. Some comparisons for showing the importance of parameters $\tau$, $\delta$, $\gamma$ via the vanishing cases will be proposed after demonstrating the main results (Theorem \ref{Theorem-Linear-Estimates} and Theorem \ref{Theorem-Linear-Estimates-2}).

Our second contribution is to justify the large time asymptotic profile of global in-time small data Sobolev solution to the nonlinear weakly damped JMGT equation \eqref{Eq-Dissipation-JMGT} in the sub-critical case $\delta>0$ as follows:
\begin{align}\label{Profile-Nonlinear}
\widetilde{\psi}(t,x):=c_{n,\gamma} \,M_{\tau,\gamma}\,t^{-\frac{n}{2}}\,\mathrm{e}^{-\frac{\gamma|x|^2}{4t}}
\end{align}
with the constants $c_{n,\gamma}:=\gamma^{-1}2^{-\frac{n}{2}}$ and
\begin{align}\label{M-Mom}
	M_{\tau,\gamma}:=\int_{\mb{R}^n}\left(\gamma\psi_0(x)+\psi_1(x)+\tau\psi_2(x)-\frac{\tau B}{2A}|\psi_1(x)|^2-\tau|\nabla\psi_0(x)|^2\right)\mathrm{d}x.
\end{align}
Let $j\in\{0,1,2\}$. For the Sobolev initial data with the additional $L^1$ integrability, we not only prove the large time optimal decay estimate (guaranteed by the identical time-dependent coefficient in its upper and lower bounds)
\begin{align*}
\|\partial_t^j\psi(t,\cdot)\|_{\dot{H}^{\sigma+2-j}}\approx t^{-\frac{n+2\sigma+4+2j}{4}}\ \ \mbox{if}\ \ M_{\tau,\gamma}\neq0,
\end{align*}
but also deduce the time-weighted convergence
\begin{align*}
\lim\limits_{t\to+\infty}t^{\frac{n+2\sigma+4+2j}{4}}\|\partial_t^j\psi(t,\cdot)-(-\Delta)^{j}\widetilde{\psi}(t,\cdot)\|_{\dot{H}^{\sigma+2-j}}=0,
\end{align*}
with $s>\max\{\frac{n}{2}-1,0\}$ for $\sigma\in\{j-2,s\}$ and any $n\geqslant 1$. These conclusions arising from Theorem \ref{Thm-GESDS} show the crucial effect of weak damping term $+\gamma\psi_t$ comparing with the limit case $\gamma=0$, i.e. the Cauchy problem for the classical JMGT equation \eqref{Eq-Real-JGMT-W}, in the references \cite{Racke-Said-2020,Said-Houari=Besov=2022,Said-Houari=Large-Norm=2022,Chen-Takeda=2023}. One of the major difficulties for studying \eqref{Eq-Dissipation-JMGT} is the treatment of derivative-type nonlinearity $\partial_t\ml{N}(\psi_t,\nabla\psi)$ when we estimate $\partial_t^2\psi$ in $n=1$, particularly, the slower decay rate of local term $|\nabla\psi|^2$ in comparison with the one of $|\psi_t|^2$. It can be predicted by $|\,\xi\,\mathrm{e}^{-\frac{1}{\gamma}|\xi|^2t}|^2\approx t^{-1}|\partial_t\mathrm{e}^{-\frac{1}{\gamma}|\xi|^2t}|^2$ as $|\xi|\ll 1$. To overcome this obstacle, we will introduce different representations of mild solution $\partial_t^j\psi$ between two cases $j\in\{0,1\}$ and $j=2$. One may see \eqref{Rep-non} and \eqref{Rep-non-2} later.

\subsection{Notations}
\hspace{5mm}The positive constants $c$ and $C$ may be changed from line to line but are independent of the time variable. We write $f\lesssim g$ if there exists a positive constant $C$ such that $f\leqslant Cg$, analogously for $f\gtrsim g$. The sharp relation $f\approx g$ holds if and only if $g\lesssim f\lesssim g$. We denote by $f\ast_{(x)}g$ the convolution of $f$ and $g$ with respect to $x$.

 We take the following zones:
\begin{align*}
	\ml{Z}_{\intt}(\varepsilon_0):=\{|\xi|\leqslant\varepsilon_0\ll1\}, \ \ 
	\ml{Z}_{\bdd}(\varepsilon_0,N_0):=\{\varepsilon_0\leqslant |\xi|\leqslant N_0\},\ \ 
	\ml{Z}_{\extt}(N_0):=\{|\xi|\geqslant N_0\gg1\}.
\end{align*}
Moreover, the cut-off functions $\chi_{\intt}(\xi),\chi_{\bdd}(\xi),\chi_{\extt}(\xi)\in \mathcal{C}^{\infty}$ having their supports in the corresponding zones $\ml{Z}_{\intt}(\varepsilon_0)$, $\ml{Z}_{\bdd}(\varepsilon_0/2,2N_0)$ and $\ml{Z}_{\extt}(N_0)$, respectively, satisfying
\begin{align*}
	\chi_{\bdd}(\xi)=1-\chi_{\intt}(\xi)-\chi_{\extt}(\xi)\ \ \mbox{for all}\ \ \xi \in \mb{R}^n.
\end{align*}

The differential operators $\langle D\rangle^s$ and $|D|^s$ have their symbols $\langle\xi\rangle^s$ and $|\xi|^s$, respectively, with any $s\in\mb{R}$, where we denote the Japanese bracket via $\langle\xi\rangle^2:=1+|\xi|^2$.

The mean of summable function $f_0=f_0(x)$ is denoted by $P_{f_0}:=\int_{\mb{R}^n}f_0(x)\,\mathrm{d}x$ that is also the 0-th moment condition of $f_0$.

\section{Main results}\setcounter{equation}{0}\label{Section-Main-Results}
\hspace{5mm}Before stating our main theorem (the reason for $\gamma\neq\frac{1}{4\tau}$ is addressed in Remark \ref{Rem-special-linear}), we are going to re-express the large time profile $\widetilde{\psi}=\widetilde{\psi}(t,x)$ defined in \eqref{Profile-Nonlinear}. Let us introduce the $\gamma$-dependent Gaussian function
\begin{align*}
G(t,x):=\frac{1}{\gamma}\ml{F}^{-1}_{\xi\to x}\left(\mathrm{e}^{-\frac{1}{\gamma}|\xi|^2t}\right)=c_{n,\gamma} \,t^{-\frac{n}{2}}\,\mathrm{e}^{-\frac{\gamma|x|^2}{4t}},
\end{align*}
which also can be understood by $\ml{G}(t,|D|)$ associated with its Fourier transform $\widehat{\ml{G}}(t,|\xi|):=\frac{1}{\gamma}\,\mathrm{e}^{-\frac{1}{\gamma}|\xi|^2t}$.  Moreover, we denote the combined data $\Psi_0=\Psi_0(x)$ such that
\begin{align}\label{Combined-data}
\Psi_{0}(x):=\gamma\psi_0(x)+\psi_1(x)+\tau\psi_2(x).
\end{align}
As a consequence, the function
\begin{align}\label{An-Rep}
\widetilde{\psi}(t,x)=G(t,x)P_{\Psi_0-\tau\ml{N}(\psi_1,\nabla\psi_0)}
\end{align}
can be explained by the Gaussian function multiplying the 0-th moment condition of combined data $\Psi_0$ and nonlinear data $\ml{N}(\psi_1,\nabla\psi_0)$. Note that
\begin{align*}
\|(-\Delta)^j\widetilde{\psi}(t,\cdot)\|_{\dot{H}^{\sigma+2-j}}\approx t^{-\frac{n+2\sigma+4+2j}{4}}|P_{\Psi_0-\tau\ml{N}(\psi_1,\nabla\psi_0)}|\ \ \mbox{as} \ \ t\gg1
\end{align*}
holds for any $\sigma\in\{j-2,s\}$ with $j\in\{0,1,2\}$ and $s>0$, due to the Gaussian kernel $G(t,x)$. See also Lemma \ref{Lemma-3.1}.

\begin{theorem}\label{Thm-GESDS}
Let $\tau>0$, $\delta>0$, $0<\gamma\neq\frac{1}{4\tau}$, and $\sigma\in\{j-2,s\}$ for $j\in\{0,1,2\}$. Let $(\psi_0,\psi_1,\psi_2)\in (H^{s+2}\cap L^1)\times (H^{s+1}\cap L^1)\times (H^s\cap L^1)$ with $s>\max\{\frac{n}{2}-1,0\}$ for any $n\geqslant 1$ being sufficiently small in the corresponding topology. Then, the weakly damped JMGT equation in the sub-critical case has a uniquely determined global in-time Sobolev solution
	\begin{align*}
	\psi\in\ml{C}([0,+\infty),H^{s+2})\cap \ml{C}^1([0,+\infty),H^{s+1})\cap \ml{C}^2([0,+\infty),H^s)
\end{align*}
satisfying the optimal decay estimate
\begin{align}\label{Upper-Nonlinear}
	\|\partial_t^j\psi(t,\cdot)\|_{\dot{H}^{\sigma+2-j}}\lesssim (1+t)^{-\frac{n+2\sigma+4+2j}{4}}\|(\psi_0,\psi_1,\psi_2)\|_{(H^{s+2}\cap L^1)\times (H^{s+1}\cap L^1)\times (H^s\cap L^1)}.
\end{align}
 Moreover, the refined estimate
\begin{align}\label{Asymptotic-Nonlinear}
	\|\partial_t^j\psi(t,\cdot)-(-\Delta)^{j}\widetilde{\psi}(t,\cdot)\|_{\dot{H}^{\sigma+2-j}}=o\big(t^{-\frac{n+2\sigma+4+2j}{4}}\big)
\end{align}
and the optimal lower bound estimate
\begin{align}\label{Lower-Nonlinear}
	\|\partial_t^j\psi(t,\cdot)\|_{\dot{H}^{\sigma+2-j}}\gtrsim t^{-\frac{n+2\sigma+4+2j}{4}}|M_{\tau,\gamma}|
\end{align}
hold for large time $t\gg1$, provided that $M_{\tau,\gamma}\neq0$ defined in \eqref{M-Mom}.
\end{theorem}
\begin{remark}\label{Rem-Compare-Nonlinear}
We are going to compare our result in Theorem \ref{Thm-GESDS} with the one for the classical JMGT equation \eqref{Eq-Real-JGMT-W} in the whole space $\mb{R}^n$ via the next table.
\begin{table}[http]
	\centering	
	\begin{tabular}{ccccc}
		\toprule
		Sub-critical JMGT equation  & Reference & Optimal rate  & Multiplier & Moment \\
		\midrule
		\multirow{2}{*}{Weakly damped model \eqref{Eq-Dissipation-JMGT}}& \multirow{2}{*}{Theorem \ref{Thm-GESDS}} & \multirow{2}{*}{$t^{-\frac{n}{4}}$} & \multirow{2}{*}{$\frac{1}{\gamma}\,\mathrm{e}^{-\frac{1}{\gamma}|\xi|^2t}$} & \multirow{2}{*}{$M_{\tau,\gamma}\neq0$}\\
		&&&&\\
		\midrule
		\multirow{2}{*}{Classical model \eqref{Eq-Real-JGMT-W}}  & \multirow{2}{*}{\cite{Racke-Said-2020,Chen-Takeda=2023}} & $\sqrt{t}$ if $n=1$; $\sqrt{\ln t}$ if $n=2$ & \multirow{2}{*}{$\frac{\sin(|\xi|t)}{|\xi|}\,\mathrm{e}^{-\frac{\delta}{2}|\xi|^2t}$}& \multirow{2}{*}{$M_{\tau,0}\neq0$}\\
		&&$t^{\frac{1}{2}-\frac{n}{4}}$ if $n\geqslant 3$&&\\
		\bottomrule
		\multicolumn{5}{l}{$\star$ The optimal estimates and asymptotic profiles for large time are served for the solution itself}
	\end{tabular}
	\caption{A comparison of sharp large time behavior for the sub-critical JMGT equations}
	\label{Table+3}
\end{table}

\noindent We notice from Table \ref{Table+3} that the weak damping term $+\gamma\psi_t$ eliminates the large time instability when $n\leqslant 2$, and improves the optimal decay rate by $t^{-\frac{1}{2}}$ when $n\geqslant 3$, from the classical case $\gamma=0$. The diffusion-wave profile as $\gamma=0$ changes into the $\gamma$-dependent Gaussian diffusion profile as $\gamma>0$ for large time $t\gg1$. Moreover, the new factor $+\gamma\psi_0(x)$ in the moment condition $M_{\tau,\gamma}\neq0$ arises. However, their assumptions for the regularities of initial data are the same.
\end{remark}

\begin{remark}
	By some interpolations, e.g. Lemma \ref{fractionalgagliardonirenbergineq}, one may derive $L^p$ estimates for the global in-time solution as a byproduct. For example,
	\begin{align*}
		\|\psi(t,\cdot)\|_{L^p}&\lesssim\|\psi(t,\cdot)\|_{L^2}^{1-\frac{n}{2(s+2)}(1-\frac{2}{p})}\|\psi(t,\cdot)\|_{\dot{H}^{s+2}}^{\frac{n}{2(s+2)}(1-\frac{2}{p})}\\
		&\lesssim (1+t)^{-\frac{n}{2}(1-\frac{1}{p})}\|(\psi_0,\psi_1,\psi_2)\|_{(H^{s+2}\cap L^1)\times (H^{s+1}\cap L^1)\times (H^s\cap L^1)}
	\end{align*}
	via the optimal estimate \eqref{Upper-Nonlinear} with $\sigma\in\{-2,s\}$, holds for
	\begin{align*}
		2\leqslant p\leqslant\begin{cases}
			+\infty&\mbox{if}\ \ n\leqslant 2(s+2),\\[0.5em]
			\displaystyle{\frac{2n}{n-2(s+2)}}&\mbox{if}\ \ n>2(s+2).
		\end{cases}
	\end{align*}
	Due to the optimality of used estimate \eqref{Upper-Nonlinear}, one may expect that the last $L^p$ estimate is sharp.
	The restriction on the exponent $p$ comes from the application of fractional Gagliardo-Nirenberg inequality, i.e. $\frac{n}{2(s+2)}(1-\frac{2}{p})\in[0,1]$.
\end{remark}

\begin{remark}
Our approach in proving Theorem \ref{Thm-GESDS} can be applied to derive similar results for the sub-critical Westervelt-type model as follows:
\begin{align}\label{Westervelt}
\tau\psi_{ttt}+\psi_{tt}-\Delta\psi-(\delta+\tau)\Delta\psi_t+\gamma\psi_t=\left(1+\frac{B}{2A}\right)\partial_t(|\psi_t|^2)\ \ \mbox{with}\ \ \delta>0
\end{align}
 in the whole space $\mb{R}^n$ for any $n\geqslant 1$, where we neglected the local nonlinear effect by the substitution corollary $|\nabla\psi|^2= |\psi_t|^2$ allowed in the weakly nonlinear scheme (basing on the shape of acoustic field being “close” to that of plane wave). This kind of nonlinearity without $|\nabla\psi|^2$ can be treated easily due to the faster decay rate predicted by $|\,\xi\,\mathrm{e}^{-\frac{1}{\gamma}|\xi|^2t}|^2\approx t^{-1}|\partial_t\mathrm{e}^{-\frac{1}{\gamma}|\xi|^2t}|^2$ as $|\xi|\ll 1$, where the Fourier multiplier derives from $G(t,x)$. So, one may directly use \eqref{Cn} instead of \eqref{Rep-non-2}, which will be explained in Remark \ref{Rem-W-model}.
\end{remark}

The other main results contributing to large time asymptotic behavior (including optimal estimates, and profiles) for its linearized Cauchy problem \eqref{Eq-Dissipation-MGT} in the sub-critical case $\delta>0$ and the critical case $\delta=0$ will be stated in Theorem \ref{Theorem-Linear-Estimates} and Theorem \ref{Theorem-Linear-Estimates-2}, respectively. We also show a large time convergence of solution from the sub-critical case to the critical case.

Nevertheless, global in-time behavior for the weakly damped JMGT equation \eqref{Eq-Dissipation-JMGT} in the critical case $\delta=0$ remains open due to some difficulties from the regularity-loss decay property as well as the higher-order derivatives in $\partial_t\ml{N}(\psi_t,\nabla\psi)$.

\section{Weakly damped MGT equation}\setcounter{equation}{0}\label{Section-Linear}
\hspace{5mm}As our preparations for studying the nonlinear model and revealing underlying physical phenomena, this section principally contributes to large time asymptotic behavior for the corresponding linearized model to \eqref{Eq-Dissipation-JMGT}, namely,
\begin{align}\label{Eq-Dissipation-MGT}
	\begin{cases}
		\tau\varphi_{ttt}+\varphi_{tt}-\Delta\varphi-(\delta+\tau)\Delta\varphi_t+\gamma\varphi_t=0,&x\in\mb{R}^n,\ t>0,\\
		\varphi(0,x)=\varphi_0(x),\ \varphi_t(0,x)=\varphi_1(x),\ \varphi_{tt}(0,x)=\varphi_2(x),&x\in\mb{R}^n,
	\end{cases}
\end{align}
with $\tau>0$, $\delta\geqslant0$, and $\gamma>0$ but $\gamma\neq\frac{1}{4\tau}$, whose reason is stated in Remark \ref{Rem-special-linear}. This section is divided into two parts. We in the first part apply the WKB analysis to derive asymptotic behavior of solutions in the Fourier space according to the values of parameters $\tau,\delta,\gamma$, and the sizes of frequencies. Then, in the second part, some qualitative properties (including well-posedness, optimal decay estimates with/without regularity-loss, and asymptotic profiles) of solutions are obtained, where some detailed comparisons with the singular limit case $\tau=0$, or with the classical MGT equation $\gamma=0$ are addressed via several remarks.
\subsection{Asymptotic behavior of solutions in the Fourier space}\label{Sub-Section-Fourier-space}
\hspace{5mm}As usual, applying the partial Fourier transform with respect to $x\in\mb{R}^n$ to the linear Cauchy problem \eqref{Eq-Dissipation-MGT}, we thus deduce
\begin{align}\label{Eq-Dissipation-GT}
	\begin{cases}
		\tau\widehat{\varphi}_{ttt}+\widehat{\varphi}_{tt}+[\gamma+(\delta+\tau)|\xi|^2]\widehat{\varphi}_t+|\xi|^2\widehat{\varphi}=0,&\xi\in\mb{R}^n,\ t>0,\\
		\widehat{\varphi}(0,\xi)=\widehat{\varphi}_0(\xi),\ \widehat{\varphi}_t(0,\xi)=\widehat{\varphi}_1(\xi),\ \widehat{\varphi}_{tt}(0,\xi)=\widehat{\varphi}_2(\xi),&\xi\in\mb{R}^n,
	\end{cases}
\end{align}
whose characteristic equation is given by
\begin{align}\label{Eq-characteristic}
	\tau\lambda^3+\lambda^2+[\gamma+(\delta+\tau)|\xi|^2] \lambda+|\xi|^2=0.
\end{align}
The discriminant characterizing the roots of this cubic is 
\begin{align*}
	\triangle_{\mathrm{Dis}}(|\xi|)= \begin{cases}
		-3\gamma^2(1-4\gamma\tau)+O(|\xi|^2)&\text{if}\ 
		\ \xi\in\ml{Z}_{\intt}(\varepsilon_0) , \\
		12\tau(\delta+\tau)^3|\xi|^6+O(|\xi|^4)&\text{if}\ 
		\ \xi\in\ml{Z}_{\extt}(N_0) .
	\end{cases}
\end{align*}
\subsubsection{Asymptotic expansions for the characteristic roots}
\hspace{5mm}The characteristic roots are heavily influenced by the parameters $\tau, \delta, \gamma$, and the sizes of frequencies. Therefore, we are going to discuss them in three cases (relying on the frequencies) with further sub-cases (relying on the parameters).
\begin{remark}\label{Rem-special-linear}
	Concerning the limit case $\gamma=\frac{1}{4\tau}$ for small frequencies $\xi\in\ml{Z}_{\intt}(\varepsilon_0)$, the $|\xi|^0$-term in the discriminant vanishes, which leads to $\triangle_{\mathrm{Dis}}(|\xi|)>0$ if $\delta>\tau$ and $\triangle_{\mathrm{Dis}}(|\xi|)<0$ if $\delta\leqslant\tau$. In other words, one just needs to follow the same procedures as those in the cases $\gamma\neq\frac{1}{4\tau}$ later via similar but tedious and lengthy computations. For briefness, we do not consider this limit case in the present work to clarify concisely the influence of external weak damping term.
\end{remark}

\noindent\textbf{Case 1: Small Frequencies $\xi\in\ml{Z}_{\intt}(\varepsilon_0)$.}
Three roots for $j\in\{1,2,3\}$ can be expanded by 
\begin{align*}
	\lambda_j=\sum_{k=0}^{+\infty}\lambda_{j,k}|\xi|^k \ \ \text{with} \ \ \lambda_{j,k}\in\mb{C}.
\end{align*}
Note that the sign of discriminant is determined by the relation between $\gamma$ and $\tau$. 
\begin{itemize}
	\item {\bf Case 1.1: $\gamma>\frac{1}{4\tau}\Rightarrow\triangle_\mathrm{Dis}(|\xi|)>0$.} The cubic \eqref{Eq-characteristic} has a real root $\lambda_1$ and a pair of non-real complex conjugate roots $\lambda_{2,3}$ as follows: 
	\begin{align*}
		\lambda_1=-\frac{1}{\gamma}|\xi|^2+O(|\xi|^4),\ \ 
		\lambda_{2,3}=-\frac{1}{2\tau}\pm i\frac{\sqrt{4\gamma\tau-1}}{2\tau}+O(|\xi|^2).
	\end{align*}
	\item {\bf Case 1.2: $\gamma<\frac{1}{4\tau}\Rightarrow\triangle_{\mathrm{Dis}}(|\xi|)<0$.}  The cubic \eqref{Eq-characteristic} has three distinct real roots $\lambda_1$ and $\lambda_{2,3}$ as follows:
	\begin{align*}
		\lambda_1=-\frac{1}{\gamma}|\xi|^2+O(|\xi|^4),\ \ 
		\lambda_{2,3}=-\frac{1}{2\tau}\pm\frac{\sqrt{1-4\gamma\tau}}{2\tau}+O(|\xi|^2).
	\end{align*}
\end{itemize}

\noindent\textbf{Case 2: Large Frequencies $\xi\in\ml{Z}_{\extt}(N_0)$.}
Thanks to the fact that $\triangle_{\mathrm{Dis}}(|\xi|)>0$, the cubic \eqref{Eq-characteristic} has a real root $\lambda_1$ and a pair of non-real complex conjugate roots $\lambda_{2,3}$ that can be expanded for $j\in\{1,2,3\}$ by 
\begin{align*}
	\lambda_j=\bar{\lambda}_{j,1}|\xi|+\sum_{k=0}^{+\infty}\bar{\lambda}_{j,-k}|\xi|^{-k} \ \text{ with } \  \bar{\lambda}_{j,1},\bar{\lambda}_{j,-k}\in\mb{C}.
\end{align*}
To understand the influence of parameter $\delta$, similarly to the classical MGT equation without the additional weak damping term, we next separate our discussion into the sub-critical case $\delta>0$ and the critical case $\delta=0$. 
\begin{remark}
The super-critical case $\delta<0$ is instable in the sense that $\mathrm{Re}\,\lambda_{2}>0$, precisely, $\mathrm{Re}\,\lambda_2=\sqrt{-\frac{\delta+\tau}{\tau}}\,|\xi|$ if $\delta+\tau<0$; $\mathrm{Re}\,\lambda_{2}=2^{-1}\tau^{-\frac{1}{3}}|\xi|^{\frac{2}{3}}$ if $\delta+\tau=0$; $\mathrm{Re}\,\lambda_2=\frac{-\delta}{2\tau(\delta+\tau)}$ if $\delta+\tau>0$.
\end{remark}
\begin{itemize}
	\item {\bf Case 2.1: $\delta>0$.} Let us expand the roots through the Laurent series to deduce
	\begin{align*}
		\lambda_1=-\frac{1}{\delta+\tau}+O(|\xi|^{-1}),\ \ 
		\lambda_{2,3}=\pm i\sqrt{\frac{\delta}{\tau}+1}\ |\xi|-\frac{\delta}{2\tau(\delta+\tau)}+O(|\xi|^{-1}).
	\end{align*}
	\item {\bf Case 2.2: $\delta=0$.} Different from Case 2.1, the $|\xi|^0$-term in $\lambda_{2,3}$ vanishes (we thus require their non-trivial real parts) implying
	\begin{align*}
		\lambda_1=-\frac{1}{\tau}+O(|\xi|^{-1}),\ \ 
		\lambda_{2,3}=\pm i|\xi|\pm i\frac{\gamma}{2\tau}|\xi|^{-1}-\frac{\gamma}{2\tau^2}|\xi|^{-2}+O(|\xi|^{-3}).
	\end{align*}
	Notice that $\lambda_{2,3}|_{\gamma=0}=\pm i|\xi|$, namely, $\mathrm{Re}\,\lambda_{2,3}|_{\gamma=0}=0$. This step preliminarily reveals the significant difference between these two cases depending on $\gamma$.
\end{itemize}

\noindent\textbf{Case 3: Bounded Frequencies $\xi\in\ml{Z}_{\bdd}(\varepsilon_0,N_0)$.}
We will argue by contradiction. There exists $j_0\in\{1,2,3\}$ such that $\lambda_{j_0}=ia_{j_0}$ with $a_{j_0}\in\mb{R}\backslash\{0\}$ which means \eqref{Eq-characteristic} has a pure imaginary root. It satisfies
\begin{align*}
	-i\tau a_{j_0}^3-a_{j_0}^2+i[\gamma+(\delta+\tau)|\xi|^2]a_{j_0}+|\xi|^2=0,
\end{align*}
which implies a contradiction according to
\begin{align*}
\begin{cases}
	-\tau a_{j_0}^3+[\gamma+(\delta+\tau)|\xi|^2] a_{j_0}=0,\\
	-a_{j_0}^2+|\xi|^2=0,
\end{cases}	\Rightarrow\ \ |\xi|^2=a_{j_0}^2=\frac{1}{\tau}[\gamma+(\delta+\tau)|\xi|^2].
\end{align*}
Thus, the roots of \eqref{Eq-characteristic} cannot be pure imaginary. Recalling that $\mathrm{Re}\,\lambda_j<0$ for $\xi\in\ml{Z}_{\intt}(\varepsilon_0)\cup\ml{Z}_{\extt}(N_0)$ in the last two cases and using the continuity of characteristic roots with respect to $|\xi|$, we claim immediately $\mathrm{Re}\,\lambda_j<0$ for all $j\in\{1,2,3\}$ as $\xi\in\ml{Z}_{\bdd}(\varepsilon_0,N_0)$.

\subsubsection{Asymptotic representations and pointwise estimates of solutions}
\hspace{5mm}Due to the previous expansions in different cases, one has to estimate the solution $\widehat{\varphi}$ as well as its time-derivatives in the corresponding circumstances.
\medskip

\noindent\textbf{Estimates in Case 1.1.}
Thanks to the distinct characteristic roots, the solution to the Cauchy problem \eqref{Eq-Dissipation-GT} is uniquely expressed via
\begin{align}\label{expantion of root1}
	\widehat{\varphi}=\widehat{K}_0\widehat{\varphi}_0+\widehat{K}_1\widehat{\varphi}_1+\widehat{K}_2\widehat{\varphi}_2
\end{align}
equipping the kernels in the Fourier space $\widehat{K}_{l}=\widehat{K}_{l}(t,|\xi|)$ for $l\in\{0,1,2\}$ as follows:
\begin{align*}
	\widehat{K}_0:=\sum_{j\in\{1,2,3\}}\frac{\mathrm{e}^{\lambda_jt}\prod_{k\neq j}\lambda_k}{\prod_{k\neq j}(\lambda_j-\lambda_k)},\ \ 
	\widehat{K}_1:=-\sum_{j\in\{1,2,3\}}\frac{\mathrm{e}^{\lambda_jt}\sum_{k\neq j}\lambda_k}{\prod_{k\neq j}(\lambda_j-\lambda_k)},\ \ 
	\widehat{K}_2:=\sum_{j\in\{1,2,3\}}\frac{\mathrm{e}^{\lambda_jt}}{\prod_{k\neq j}(\lambda_j-\lambda_k)}.
\end{align*}
Note that $k\in\{1,2,3\}$ in the last formulas.
The pair of complex conjugate roots can be re-expressed via $\lambda_{2,3}=\lambda_{\mathrm{RS}}\pm i\lambda_{\mathrm{IS}}$ (the subscript ``S'' means the situation of small frequencies), in which we denote
\begin{align}\label{1.1root}
	\lambda_{\mathrm{RS}}=-\frac{1}{2\tau}+O(|\xi|^2)\ \ \text{and}\ \  \lambda_{\mathrm{IS}}=\frac{\sqrt{4\gamma\tau-1}}{2\tau}+O(|\xi|^2).
\end{align}
Benefiting from the conjugate structure of $\lambda_{2,3}$ and the Euler formula, we may rewrite the solution's expression \eqref{expantion of root1} by
\begin{align}\label{expantion of root2}
	\widehat{\varphi}&=\frac{({\lambda^2_{\mathrm{IS}}}+\lambda_{\mathrm{RS}}^2)\widehat{\varphi}_0-2\lambda_{\mathrm{RS}}\widehat{\varphi}_1+\widehat{\varphi}_2}{(\lambda_1-\lambda_{\mathrm{RS}})^2+\lambda_{\mathrm{IS}}^2}\,\mathrm{e}^{\lambda_1t}+
	\frac{(\lambda_1-2\lambda_{\mathrm{RS}})\lambda_1\widehat{\varphi}_0+2\lambda_{\mathrm{RS}}\widehat{\varphi}_1-\widehat{\varphi}_2}{(\lambda_1-\lambda_{\mathrm{RS}})^2+\lambda_{\mathrm{IS}}^2}\cos(\lambda_{\mathrm{IS}}t)\,\mathrm{e}^{\lambda_{\mathrm{RS}}t}\notag\\
	&\quad\ +\frac{[\lambda_{\mathrm{RS}}(\lambda_{\mathrm{RS}}-\lambda_1)-\lambda_{\mathrm{IS}}^2]\lambda_1\widehat{\varphi}_0+(\lambda_1^2-\lambda_{\mathrm{RS}}^2+\lambda_{\mathrm{IS}}^2)\widehat{\varphi}_1+(\lambda_{\mathrm{RS}}-\lambda_1)\widehat{\varphi}_2}{\lambda_{\mathrm{IS}}[(\lambda_1-\lambda_{\mathrm{RS}})^2+\lambda_{\mathrm{IS}}^2]}\sin(\lambda_{\mathrm{IS}}t)\,\mathrm{e}^{\lambda_{\mathrm{RS}}t}.
\end{align}
Plugging \eqref{1.1root} and $\lambda_1=-\frac{1}{\gamma}|\xi|^2+O(|\xi|^4)$ into \eqref{expantion of root2}, by straightforward calculations, one arrives at
\begin{align}\label{es1.1.1}
	\chi_{\intt}(\xi)|\partial_t^j\widehat{\varphi}|&\lesssim\chi_{\intt}(\xi)\left(|\xi|^{2j}\,\mathrm{e}^{-c|\xi|^2t}+\mathrm{e}^{-ct}\right)(|\widehat{\varphi}_0|+|\widehat{\varphi}_1|+|\widehat{\varphi}_2|)
\end{align}
for all $j\in\{0,\dots,3\}$.
Furthermore, let us recall the Fourier multiplier $\widehat{\ml{G}}(t,|\xi|)=\frac{1}{\gamma}\,\mathrm{e}^{-\frac{1}{\gamma}|\xi|^2t}$, and introduce the combined data in the Fourier space via $\widehat{\Phi}_{0}:=\gamma\widehat{\varphi}_0+\widehat{\varphi}_1+\tau\widehat{\varphi}_2$ similarly to \eqref{Combined-data}. Analogously to the last estimate, we deduce that
\begin{align}\label{es1.1.2}
	\chi_{\intt}(\xi)|\partial_t^j\widehat{\varphi}-|\xi|^{2j}\widehat{\ml{G}}(t,|\xi|)\widehat{\Phi}_{\mathrm{0}}|
	&\lesssim\chi_{\intt}(\xi)\big(|\xi|^{2j+2}\,\mathrm{e}^{-c|\xi|^2t}+\mathrm{e}^{-ct}\big)(|\widehat{\varphi}_0|+|\widehat{\varphi}_1|+|\widehat{\varphi}_2|)
\end{align}
for all $j\in\{0,\dots,3\}$, in which we used some error estimates in \eqref{expantion of root2} by means of subtracting the corresponding leading terms from the asymptotic analysis, for instance,
\begin{align*}
	&\chi_{\intt}(\xi)\left|\frac{\widehat{\varphi}_2}{(\lambda_1-\lambda_{\mathrm{RS}})^2+\lambda_{\mathrm{IS}}^2}\,\mathrm{e}^{\lambda_1t}-\widehat{\ml{G}}(t,|\xi|)\tau\widehat{\varphi}_2\right|\\
	&\lesssim\chi_{\intt}(\xi)\left|\frac{1}{\frac{\gamma}{\tau}+O(|\xi|^2)}\,\mathrm{e}^{-\frac{1}{\gamma}|\xi|^2t+O(|\xi|^4)t}-\frac{\tau}{\gamma}\,\mathrm{e}^{-\frac{1}{\gamma}|\xi|^2t}\right||\widehat{\varphi}_2| \\
	&\lesssim\chi_{\intt}(\xi)\left(\frac{|\xi|^2}{\frac{\gamma}{\tau}[\frac{\gamma}{\tau}+O(|\xi|^2)]}\,\mathrm{e}^{-c|\xi|^2t}+|\xi|^4t\,\mathrm{e}^{-\frac{1}{\gamma}|\xi|^2t}\int_0^1\mathrm{e}^{O(|\xi|^4)t\eta}\,\mathrm{d}\eta\right)|\widehat{\varphi}_2|\\
	&\lesssim\chi_{\intt}(\xi)|\xi|^2\,\mathrm{e}^{-c|\xi|^2t}\,|\widehat{\varphi}_2|.
\end{align*}
Here, we employed $|\xi|^4t\,\mathrm{e}^{-\frac{1}{2\gamma}|\xi|^2t}\lesssim|\xi|^2\,\mathrm{e}^{-\frac{c}{2}|\xi|^2t}$ and $\mathrm{e}^{-\frac{1}{2\gamma}|\xi|^2t}\int_0^1\mathrm{e}^{O(|\xi|^4)t\eta}\,\mathrm{d}\eta\lesssim \mathrm{e}^{-\frac{c}{2}|\xi|^2t}$ for $\xi\in\ml{Z}_{\intt}(\varepsilon_0)$.
\medskip

\noindent\textbf{Estimates in Case 1.2.} Although three distinct characteristic roots are real (different from those with the conjugate structure in Case 1.1), we still are able to plug the asymptotic expansions of these roots into the general representation \eqref{expantion of root1} to derive the estimates \eqref{es1.1.1} and \eqref{es1.1.2} for all $j\in\{0,\dots,3\}$. The derivation of refined estimate \eqref{es1.1.2} needs to subtract the corresponding leading terms motivated by the asymptotic analysis, for example,
\begin{align*}
&\chi_{\intt}(\xi)|\widehat{K}_0(t,|\xi|)\widehat{\varphi}_0-\widehat{\ml{G}}(t,|\xi|)\gamma\widehat{\varphi}_0|\\
&\lesssim\chi_{\intt}(\xi)\left(\,\left|\frac{\frac{\gamma}{\tau}+O(|\xi|^2)}{\frac{\gamma}{\tau}+O(|\xi|^2)}\,\mathrm{e}^{-\frac{1}{\gamma}|\xi|^2t+O(|\xi|^4)t}-\mathrm{e}^{-\frac{1}{\gamma}|\xi|^2t}\right|+|\xi|^2\,\mathrm{e}^{-ct}\right)|\widehat{\varphi}_0|\\
&\lesssim\chi_{\intt}(\xi)\big(|\xi|^{2}\,\mathrm{e}^{-c|\xi|^2t}+\mathrm{e}^{-ct}\big)|\widehat{\varphi}_0|.
\end{align*}
Note that the last two terms $O(|\xi|^2)$  are different.
Let us finally explain the reason for the same estimates between Case 1.1 and Case 1.2. The dominant parts in the leading terms of $\partial_t^j\widehat{\varphi}$ are generated by the first characteristic root $\lambda_1\sim-\frac{1}{\gamma}|\xi|^2$ as $|\xi|\to0$, which are identical. The different parts mainly come from the $|\xi|^0$-terms of $\lambda_2$ and $\lambda_3$. However, their dominant real parts being strictly negative imply exponential decay factors $\mathrm{e}^{-\frac{1}{2\tau}t}$ in Case 1.1 and $\mathrm{e}^{-(\frac{1}{2\tau}\mp\frac{\sqrt{1-4\gamma\tau}}{2\tau})t}$ in Case 1.2 that always are estimated by regular in-frequency remainders $|\xi|^{p\geqslant0}$ with $\mathrm{e}^{-ct}$. Consequently, we have derived the same leading term $|\xi|^{2j}\widehat{\ml{G}}(t,|\xi|)\widehat{\Phi}_0$.
\medskip

\noindent\textbf{Estimates in Case 2.1.} 
Let us re-expressed the pair of complex conjugate roots via $\lambda_{2,3}=\lambda_{\mathrm{RL}}\pm i\lambda_{\mathrm{IL}}$ (the subscript ``L'' means the situation of large frequencies) carrying
\begin{align*}
	\lambda_{\mathrm{RL}}=-\frac{\delta}{2\tau(\delta+\tau)}+O(|\xi|^{-1})\ \ \mbox{and}\ \ 
	\lambda_{\mathrm{IL}}=\sqrt{\frac{\delta}{\tau}+1}\ |\xi|+O(|\xi|^{-1}).
\end{align*}
We still may use the representation \eqref{expantion of root2} after replacing $\lambda_{\mathrm{IS}}$ [resp. $\lambda_{\mathrm{RS}}]$ by $\lambda_{\mathrm{IL}}$ [resp. $\lambda_{\mathrm{RL}}]$. Consequently, associated with the asymptotic expansions of these roots, one computes directly
\begin{align}\label{es2.1.1}
	\chi_{\extt}(\xi)|\partial_t^j\widehat{\varphi}|&\lesssim\chi_{\extt}(\xi)\,\mathrm{e}^{-ct}(|\xi|^{\max\{j-1,0\}}|\widehat{\varphi}_0|+|\xi|^{j-1}|\widehat{\varphi}_1|+|\xi|^{j-2}|\widehat{\varphi}_2|)
\end{align}
for all $j\in\{0,\dots,3\}$. Note that we do not consider an error estimate because of its exponential decay factor $\mathrm{e}^{-ct}$ independent of $|\xi|$.
\medskip

\noindent\textbf{Estimates in Case 2.2.}
Analogously to Case 2.1, the real and imaginary parts of $\lambda_{2,3}$ are denoted, respectively, by
\begin{align*}
	 \lambda_{\mathrm{RL}}=-\frac{\gamma}{2\tau^2}|\xi|^{-2}+O(|\xi|^{-3})\ \ \mbox{and}\ \  \lambda_{\mathrm{IL}}=|\xi|+\frac{\gamma}{2\tau}|\xi|^{-1}+O(|\xi|^{-3}).
\end{align*}
Let us now plug the updated asymptotic expansions of three roots into the representation \eqref{expantion of root2}. Therefore, by straightforward computations, one obtains
\begin{align}\label{es2.2.1}
	\chi_{\extt}(\xi)|\partial_t^j\widehat{\varphi}|&\lesssim\chi_{\extt}(\xi)\,\mathrm{e}^{-c|\xi|^{-2}t}(|\xi|^{\max\{j-1,0\}}|\widehat{\varphi}_0|+|\xi|^{j-1}|\widehat{\varphi}_1|+|\xi|^{j-2}|\widehat{\varphi}_2|)
\end{align}
and 
\begin{align}\label{es2.2.2}
	&\chi_{\extt}(\xi)\big|\partial_t^j\widehat{\varphi}-\mathrm{e}^{-\frac{\gamma}{2\tau^2}|\xi|^{-2}t}\big[\partial_t^j\sin(|\xi|t)\big(\tfrac{1}{\tau^2}|\xi|^{-1}\widehat{\varphi}_0+|\xi|^{-1}\widehat{\varphi}_1\big)-\partial_t^j\cos(|\xi|t)|\xi|^{-2}\widehat{\varphi}_2\big]\big|\notag\\
	&\lesssim\chi_{\extt}(\xi)|\xi|^{-1}\,\mathrm{e}^{-c|\xi|^{-2}t}(|\xi|^{\max\{j-1,0\}}|\widehat{\varphi}_0|+|\xi|^{j-1}|\widehat{\varphi}_1|+|\xi|^{j-2}|\widehat{\varphi}_2|)
\end{align}
for all $j\in\{0,\dots,3\}$.

\medskip

\noindent\textbf{Estimates in Case 3.}
Due to the negative real parts of all characteristic roots, it is easy to claim the following exponential decay estimate without asking for any Sobolev regularity of initial data:
\begin{align}\label{es3}
	\chi_{\bdd}(\xi)|\partial_t^j\widehat{\varphi}|\lesssim\chi_{\bdd}(\xi)\,\mathrm{e}^{-ct}(|\widehat{\varphi}_0|+|\widehat{\varphi}_1|+|\widehat{\varphi}_2|)
\end{align}
for all $j\in\{0,\dots,3\}$.

\subsection{Large time qualitative properties of solutions}\label{Sub-Section-Linear-Properties}
\hspace{5mm}It is well-known that the regularity of well-posedness is completely determined by the property of solutions for large frequencies. Therefore, by using \eqref{es2.1.1} if $\delta>0$ and \eqref{es2.2.1} if $\delta=0$, for all $j\in\{0,1,2\}$ and any $s\in\mb{R}$, it is not difficult to prove
\begin{align*}
\chi_{\extt}(\xi)\langle\xi\rangle^{s+2-j}|\partial_t^j\widehat{\varphi}|\lesssim \chi_{\extt}(\xi)(\langle\xi\rangle^{s+2}|\widehat{\varphi}_0|+\langle\xi\rangle^{s+1}|\widehat{\varphi}_1|+\langle\xi\rangle^{s}|\widehat{\varphi}_2|),
\end{align*}
which immediately concludes the next statement.
\begin{prop}\label{PROP-3.1}
	Let $(\varphi_0,\varphi_1,\varphi_2)\in H^{s+2}\times H^{s+1}\times H^s$ with $s\in\mb{R}$. Then, the weakly damped MGT equation \eqref{Eq-Dissipation-MGT} with $\delta\geqslant0$ has a uniquely determined global in-time Sobolev solution
	\begin{align*}
		\varphi\in\ml{C}([0,+\infty),H^{s+2})\cap \ml{C}^1([0,+\infty),H^{s+1})\cap \ml{C}^2([0,+\infty),H^s).
	\end{align*}
\end{prop}

It is interesting to study further fine properties of solutions after getting the global in-time well-posedness. To investigate large time asymptotic behavior, we assume the additional $L^1$ integrability (as well as some higher regularities when $\delta=0$) of Cauchy data. Before doing this, let us introduce some well-known preliminaries.
\begin{lemma}\label{Lemma-3.1}
	Let $s\in\mb{R}$ and $n+2s>0$. The optimal estimate
	\begin{align*}
	\left\|\chi_{\intt}(\xi)|\xi|^s\,\mathrm{e}^{-c|\xi|^2t}\right\|_{L^2}\approx t^{-\frac{n+2s}{4}}
	\end{align*}
	holds for large time $t\gg1$, which is bounded for $t\leqslant 1$.
\end{lemma}
\begin{lemma}\label{Lemma-3.2}
Let $s\in\mb{R}$ and $\ell\geqslant 0$. The sharp upper bound estimate
\begin{align*}
\left\|\chi_{\extt}(\xi)|\xi|^s\,\mathrm{e}^{-c|\xi|^{-2}t}\widehat{f}_0\right\|_{L^2}\lesssim (1+t)^{-\frac{\ell}{2}}\|f_0\|_{H^{s+\ell}}
\end{align*}
holds for $f_0\in H^{s+\ell}$.
\end{lemma}
\subsubsection{The sub-critical case $\delta>0$}
\hspace{5mm}Let us recall $\Phi_0=\Phi_0(x)$ such that $\Phi_{0}(x):=\gamma\varphi_0(x)+\varphi_1(x)+\tau\varphi_2(x)$, and state the main result, i.e. the optimal $(L^2\cap L^1)-L^2$ estimate and the asymptotic profile as large time, for the linear problem \eqref{Eq-Dissipation-MGT} in the sub-critical case $\delta>0$.
\begin{theorem}\label{Theorem-Linear-Estimates}
Let $\tau>0$, $\delta>0$, $0<\gamma\neq\frac{1}{4\tau}$, and $n+2s+4+2j>0$ for $j\in\{0,1, 2\}$. Let $(\varphi_0,\varphi_1,\varphi_2)\in (H^{s+2}\cap L^1)\times (H^{s+1}\cap L^1)\times (H^s\cap L^1)$. Then, the Sobolev solution to the weakly damped MGT equation \eqref{Eq-Dissipation-MGT} in the sub-critical case satisfies
\begin{align*}
\|\partial_t^j\varphi(t,\cdot)\|_{\dot{H}^{s+2-j}}\lesssim (1+t)^{-\frac{n+2s+4+2j}{4}}\|(\varphi_0,\varphi_1,\varphi_2)\|_{(H^{s+\max\{1,2-j\}}\cap L^1)\times (H^{s+1}\cap L^1)\times (H^s\cap L^1)},
\end{align*}
which is also valid for the inhomogeneous Sobolev data. Moreover, the refined estimate
\begin{align*}
\|\partial_t^j\varphi(t,\cdot)-(-\Delta)^{j}G(t,\cdot)P_{\Phi_0}\|_{\dot{H}^{s+2-j}}=o\big(t^{{-\frac{n+2s+4+2j}{4}}}\big)
\end{align*}
 and the optimal lower bound estimate
\begin{align*}
\|\partial_t^j\varphi(t,\cdot)\|_{\dot{H}^{s+2-j}}\gtrsim t^{-\frac{n+2s+4+2j}{4}}|P_{\Phi_0}|
\end{align*}
hold for large time $t\gg1$, provided that $P_{\Phi_0}\neq0$.
\end{theorem}
\begin{proof}
By applying the Hausdorff-Young inequality and Lemma \ref{Lemma-3.1} with $n+2s+4+2j>0$ for $j\in\{0,1,2\}$, we may obtain
	\begin{align*}
	\|\partial_t^j\varphi(t,\cdot)\|_{\dot{H}^{s+2-j}}&\lesssim \|\chi_{\intt}(\xi)|\xi|^{s+2-j}\partial_t^j\widehat{\varphi}(t,\xi)\|_{L^2}+\|(1-\chi_{\intt}(\xi))|\xi|^{s+2-j}\partial_t^j\widehat{\varphi}(t,\xi)\|_{L^2}\\
	&\lesssim\big\|\chi_{\intt}(\xi)|\xi|^{s+2-j}\big(|\xi|^{2j}\,\mathrm{e}^{-c|\xi|^2t}+\mathrm{e}^{-ct}\big)\big\|_{L^2}\|(\widehat{\varphi}_0,\widehat{\varphi}_1,\widehat{\varphi}_2)\|_{(L^\infty)^3}\\
	&\quad+\mathrm{e}^{-ct}\|(1-\chi_{\intt}(D))|D|^{s+2-j}(\varphi_0,\varphi_1,\varphi_2)\|_{\dot{H}^{\max\{j-1,0\}}\times \dot{H}^{j-1}\times \dot{H}^{j-2}}\\
	&\lesssim (1+t)^{-\frac{n+2s+4+2j}{4}}\|(\varphi_0,\varphi_1,\varphi_2)\|_{(H^{s+\max\{1,2-j\}}\cap L^1)\times (H^{s+1}\cap L^1)\times (H^s\cap L^1)},
\end{align*}
where we used \eqref{es1.1.1}, \eqref{es3} and \eqref{es2.1.1} in the second inequality of last chain. Analogously,
\begin{align*}
\|\chi_{\extt}(D)|D|^{2j}\ml{G}(t,|D|)\Phi_0(\cdot)\|_{\dot{H}^{s+2-j}}&\lesssim\big\|\chi_{\extt}(\xi)|\xi|^{s+2+j}\,\mathrm{e}^{-c|\xi|^2t}\,\widehat{\Phi}_0(\xi)\big\|_{L^2}\\
&\lesssim \mathrm{e}^{-ct}\|(\varphi_0,\varphi_1,\varphi_2)\|_{H^{s+\max\{1,2-j\}}\times H^{s+1} \times H^s}
\end{align*}
for large time, because
\begin{align*}
\sup\limits_{|\xi|\geqslant N_0}\big(|\xi|^{s_0}\,\mathrm{e}^{-c|\xi|^2t}\big)\lesssim t^{-\frac{s_0}{2}}\,\mathrm{e}^{-\frac{c}{2}N_0^2t}\lesssim \mathrm{e}^{-c_0t}
\end{align*}
holds for any $s_0\in\mb{R}$ and $t\gg1$.
From \eqref{es1.1.2}, we thus arrive at the large time error estimate
\begin{align}\label{p3.2.1}
	&\|\partial_t^j\varphi (t,\cdot)-|D|^{2j}\ml{G}(t,|D|)\Phi_0(\cdot)\|_{\dot{H}^{s+2-j}}\notag\\
	& \lesssim t^{-\frac{n+2s+4+2j}{4}-1}\|(\varphi_0,\varphi_1,\varphi_2)\|_{(H^{s+\max\{1,2-j\}}\cap L^1)\times (H^{s+1}\cap L^1)\times (H^s\cap L^1)}.
\end{align}
Via the Lagrange theorem
\begin{align*}
	\big||D|^{2j}G(t,x-y)-|D|^{2j}G(t,x)\big|\lesssim|y|\,\big| |D|^{2j+1}G(t,x-\theta_1y)\big|\ \ \mbox{with}\ \ \theta_1\in(0,1),
\end{align*}
one notices that
\begin{align}\label{p3.2.2}
	&\|\,|D|^{2j}\ml{G}(t,|D|)\Phi_0(\cdot)-|D|^{2j}G(t,\cdot)P_{\Phi_0}\|_{L^2}\notag\\
	&\lesssim\left\|\int_{|y|\leqslant t^{\frac{1}{3}}}\big(|D|^{2j}G(t,\cdot-y)-|D|^{2j}G(t,\cdot)\big)\Phi_0(y)\,\mathrm{d}y\right\|_{L^2}\notag\\
	&\quad+\left\|\int_{|y|\geqslant t^{\frac{1}{3}}}\big(\,\big||D|^{2j}G(t,\cdot-y)\big|+\big||D|^{2j}G(t,\cdot)\big|\,\big)|\Phi_0(y)|\,\mathrm{d}y\right\|_{L^2}\notag\\
	&\lesssim t^{\frac{1}{3}}\|\,|D|^{2j+1} G(t,\cdot)\|_{L^2}\|\Phi_0\|_{L^1}+\|\,|D|^{2j}G(t,\cdot)\|_{L^2}\|\Phi_0\|_{L^1(|x|\geqslant t^{\frac{1}{3}})}.
\end{align}
Thanks to $\Phi_0\in L^1$ from our assumption so that $\|\Phi_0\|_{L^1(|x|\geqslant t^{\frac{1}{3}})}=o(1)$ as $t\to+\infty$, we in the $\dot{H}^{s+2-j}$ norm are able to deduce
	\begin{align}\label{p3.2.3}
	&\|\,|D|^{2j}\ml{G}(t,|D|)\Phi_0(\cdot)-|D|^{2j}G(t,\cdot)P_{\Phi_0}\|_{\dot{H}^{s+2-j}}\notag\\
	&\lesssim t^{\frac{1}{3}}\|\,|\xi|^{s+3+j}\widehat{\ml{G}}(t,|\xi|)\|_{L^2}\|\Phi_0\|_{L^1}+\|\,|\xi|^{s+2+j}\widehat{\ml{G}}(t,|\xi|)\|_{L^2}\|\Phi_0\|_{L^1(|x|\geqslant t^{\frac{1}{3}})}\notag\\
	&=o\left(t^{-\frac{n+2s+4+2j}{4}}\right)
\end{align}
as $t\gg1$. By combining \eqref{p3.2.2} and \eqref{p3.2.3}, it results the desired refined estimate. For large time $t\gg1$, the Minkowski inequality and Lemma \ref{Lemma-3.1} imply that
\begin{align*}
	\|\partial_t^j\varphi(t,\cdot)\|_{\dot{H}^{s+2-j}}&\gtrsim \|\,|D|^{2j}G(t,\cdot)\|_{\dot{H}^{s+2-j}}|P_{\Phi_0}|-\|\partial_t^j\varphi(t,\cdot)-|D|^{2j}G(t,\cdot)P_{\Phi_0}\|_{\dot{H}^{s+2-j}}\\
	&\gtrsim t^{-\frac{n+2s+4+2j}{4}}|P_{\Phi_0}|- o\left(t^{-\frac{n+2s+4+2j}{4}}\right)\gtrsim t^{-\frac{n+2s+4+2j}{4}}|P_{\Phi_0}|
\end{align*}
provided that $P_{\Phi_0}\neq0$. Our proof is complete.
\end{proof}
\begin{remark}
The next optimal estimate for the additional $L^1$ integrable Cauchy data:
\begin{align*}
\|\partial_t^j\varphi(t,\cdot)\|_{\dot{H}^{s+2-j}}\approx t^{-\frac{n+2s+4+2j}{4}}
\end{align*}
holds for large time $t\gg1$, with $n+2s+4+2j>0$ for $j\in\{0,1, 2\}$, provided that the 0-th moment condition of combined data $\gamma\varphi_0(x)+\varphi_1(x)+\tau\varphi_2(x)$ is non-zero, whose key reason arises from its diffusion profile.
\end{remark}
\begin{remark}
We are going to compare our result in Theorem \ref{Theorem-Linear-Estimates} with those for the following well-studied damped models:
\begin{itemize}
	\item the double damped wave equation $\varphi^{\tau=0}_{tt}-\Delta \varphi^{\tau=0}-\delta\Delta \varphi^{\tau=0}_t+\gamma\varphi^{\tau=0}_t=0$, which is the singular limit case $\tau=0$ in \eqref{Eq-Dissipation-MGT} by setting $\delta>0$ and $\gamma>0$;
	\item the sub-critical MGT equation $\tau\varphi^{\gamma=0}_{ttt}+\varphi^{\gamma=0}_{tt}-\Delta \varphi^{\gamma=0}-(\delta+\tau)\Delta \varphi^{\gamma=0}_t=0$, which is the vanishing weak damping limit case $\gamma=0$ in \eqref{Eq-Dissipation-MGT} by setting $\tau>0$ and $\delta>0$;
\end{itemize}
in the whole space $\mb{R}^n$ via the next table.
\begin{table}[http]
	\centering	
	\begin{tabular}{cccc}
		\toprule
		Model  & Reference & Optimal rate  & Multiplier \\
		\midrule
		  Weakly damped MGT equation& \multirow{2}{*}{Theorem \ref{Theorem-Linear-Estimates}} & \multirow{2}{*}{$t^{-\frac{n}{4}}$} & \multirow{2}{*}{$\frac{1}{\gamma}\,\mathrm{e}^{-\frac{1}{\gamma}|\xi|^2t}$}\\
		\eqref{Eq-Dissipation-MGT} with $\tau>0$, $\delta>0$, $\gamma>0$&&&\\
		\midrule
		  Double damped wave equation& \multirow{2}{*}{\cite{Ikehata-Sawada=2016,Ikehata-Takeda=2017,Ikehata-Michihisa=2019,Dao-Van-Nguyen=2024}} & \multirow{2}{*}{$t^{-\frac{n}{4}}$} & \multirow{2}{*}{$\frac{1}{\gamma}\,\mathrm{e}^{-\frac{1}{\gamma}|\xi|^2t}$}\\
\eqref{Eq-Dissipation-MGT} with $\tau=0$, $\delta>0$, $\gamma>0$&&&\\
		\midrule
		Sub-critical MGT equation  & \multirow{2}{*}{\cite{Chen-Ikehata=2021,Chen-Takeda=2023,Chen-Ma-Qin=2025}} & $\sqrt{t}$ if $n=1$; $\sqrt{\ln t}$ if $n=2$; & \multirow{2}{*}{$\frac{\sin(|\xi|t)}{|\xi|}\,\mathrm{e}^{-\frac{\delta}{2}|\xi|^2t}$}\\
		\eqref{Eq-Dissipation-MGT} with $\tau>0$, $\delta>0$, $\gamma=0$&&$t^{\frac{1}{2}-\frac{n}{4}}$ if $n\geqslant 3$&\\
		\bottomrule
		\multicolumn{4}{l}{$\star$ The optimal estimates and asymptotic profiles for large time are served for the solution itself.}
	\end{tabular}
	\caption{A comparison of sharp large time behavior in the sub-critical case $\delta>0$}
	\label{Table_1}
\end{table}

\noindent We notice from Table \ref{Table_1} that the weak damping term $+\gamma\varphi_t$ plays a crucial role in the sub-critical case $\delta>0$, because the same large time behavior for $\tau\geqslant0$ but the different ones between $\gamma>0$ and $\gamma=0$. This weak damping term influences not only on the optimal rates but also on the asymptotic profiles (see the parameter $\gamma$ in the dominant kernel $G(t,x)$ and the combined data $\Phi_0(x)$). Particularly, the diffusion-wave profile changes into the $\gamma$-dependent Gaussian diffusion profile as $\gamma>0$, that eliminates the large time instability when $n\leqslant 2$, and improves the decay rate when $n\geqslant 3$, as the limit case $\gamma=0$. 
\end{remark}

As a byproduct, we address some suitable $L^2-L^2$ estimates as well as $(L^2\cap L^1)-L^2$ estimates concerning for the third kernel $K_2(t,x)$, in order to complete our nonlinear part later. Comparing with those in Theorem \ref{Theorem-Linear-Estimates}, we estimate the $(j+1)$-order time-derivatives instead of the $j$-order time-derivatives in the $\dot{H}^{s+2-j}$ norm for $j\in\{0,1,2\}$ due to the treatment of nonlinearity $\partial_t\ml{N}(\psi_t,\nabla\psi)$.

\begin{coro}\label{CORO-3.1}
Let $\tau>0$, $\delta>0$, $0<\gamma\neq\frac{1}{4\tau}$, and $n+2s+8+2j>0$ for $j\in\{0,1, 2\}$. Let $f_0\in \dot{H}^{s+1}\cap L^1$. Then, the third kernel of solution satisfies
\begin{align*}
\|\partial_t^{j+1}K_2(t,\cdot)\ast_{(x)}f_0(\cdot)\|_{\dot{H}^{s+2-j}}&\lesssim (1+t)^{-\frac{n+2s+8+2j}{4}}\|f_0\|_{\dot{H}^{\max\{s+1,0\}}\cap L^1},\\
\|\partial_t^{j+1}K_2(t,\cdot)\ast_{(x)}f_0(\cdot)\|_{\dot{H}^{s+2-j}}&\lesssim (1+t)^{-\frac{3+j}{2}}\|f_0\|_{\dot{H}^{\max\{s+1,0\}}}.
\end{align*}
\end{coro}
\begin{proof}
Their demonstrations are analogous to those in Theorem \ref{Theorem-Linear-Estimates}, particularly, the estimate for large frequencies as follows:
\begin{align*}
\|(1-\chi_{\intt}(D))\partial_t^{j+1}K_2(t,\cdot)\ast_{(x)}f_0(\cdot)\|_{\dot{H}^{s+2-j}}&\lesssim \|(1-\chi_{\intt}(\xi))|\xi|^{s+2-j}\partial_t^{j+1}\widehat{K}_2(t,|\xi|)\widehat{f}_0(\xi)\|_{L^2}\\
&\lesssim \mathrm{e}^{-ct}\|f_0\|_{\dot{H}^{\max\{s+1,0\}}}.
\end{align*}
Concerning the estimate for small frequencies, we separate into two situations.
\begin{itemize}
	\item If $f_0\in L^1$ additionally when $n+2s+8+2j>0$, then
	\begin{align*}
	\|\chi_{\intt}(D)\partial_t^{j+1}K_2(t,\cdot)\ast_{(x)}f_0(\cdot)\|_{\dot{H}^{s+2-j}}&\lesssim\big\|\chi_{\intt}(\xi)|\xi|^{s+2-j}\big(|\xi|^{2j+2}\,\mathrm{e}^{-c|\xi|^2t}+\mathrm{e}^{-ct}\big)\big\|_{L^2}\|\widehat{f}_0\|_{L^{\infty}}\\
	&\lesssim  (1+t)^{-\frac{n+2s+8+2j}{4}}\|f_0\|_{L^1}.
	\end{align*}
	\item If $f_0\in \dot{H}^{s+1}$ only, then
	\begin{align*}
	\|\chi_{\intt}(D)\partial_t^{j+1}K_2(t,\cdot)\ast_{(x)}f_0(\cdot)\|_{\dot{H}^{s+2-j}}&\lesssim\left(\,\sup\limits_{|\xi|\leqslant \varepsilon_0}\big(|\xi|^{3+j}\,\mathrm{e}^{-c|\xi|^2t}\big)+\mathrm{e}^{-ct}\right)\|\,|\xi|^{s+1}\widehat{f}_0\|_{L^2}\\
	&\lesssim (1+t)^{-\frac{3+j}{2}}\|f_0\|_{\dot{H}^{s+1}}.
	\end{align*}
\end{itemize}
Our proof is finished after summarizing the last estimates.
\end{proof}

\subsubsection{The critical case $\delta=0$}
\hspace{5mm}As a matter of fact for the weakly damped MGT equation \eqref{Eq-Dissipation-MGT} in the critical case $\delta=0$, by using the equation \eqref{Eq-Dissipation-MGT} with $\delta=0$ and the integration by parts associated with its strict hyperbolicity, the following dissipation of suitable energy:
\begin{align*}
&\frac{\mathrm{d}}{\mathrm{d}t}\left(\|\tau\varphi_{tt}(t,\cdot)+\varphi_t(t,\cdot)\|_{L^2}^2+\|\tau\nabla\varphi_t(t,\cdot)+\nabla\varphi(t,\cdot)\|_{L^2}^2+\gamma\tau\|\varphi_t(t,\cdot)\|_{L^2}^2\right)\\
&=2\int_{\mb{R}^n}\big(\tau\varphi_{ttt}+\varphi_{tt})(\tau\varphi_{tt}+\varphi_t)+(\tau\nabla\varphi_{tt}+\nabla\varphi_t)\cdot(\tau\nabla\varphi_t+\nabla\varphi)+\gamma\tau\varphi_{tt}\varphi_t\big)\,\mathrm{d}x\\
&=-2\gamma\|\varphi_t(t,\cdot)\|_{L^2}^2<0,
\end{align*}
holds for any $\gamma>0$, which is quite different from the classical situation $\gamma=0$ (the conservation of suitable energy, see \cite[Proposition 1]{Chen-Palmieri=2020}). For this reason, it is interesting to study damping phenomena, especially, the optimal $(L^2\cap L^1)-L^2$ estimate and the asymptotic profile as large time, for the linear problem \eqref{Eq-Dissipation-MGT} in the critical case $\delta=0$. Beforehand, let us denote the Fourier multipliers
\begin{align*}
\ml{H}_{\mathrm{s}}(t,|D|):=\frac{\sin(|D|t)}{|D|}\ \ \mbox{and}\ \ \ml{H}_{\mathrm{c}}(t,|D|):=\frac{\cos(|D|t)}{|D|^2}
\end{align*}
via their symbols. Moreover, we introduce another combined data $\Phi_1(x):=\frac{1}{\tau^2}\varphi_0(x)+\varphi_1(x)$ for the sake of simplicity.

\begin{theorem}\label{Theorem-Linear-Estimates-2}
	Let $\tau>0$, $\delta=0$, $0<\gamma\neq\frac{1}{4\tau}$, and $n+2s+4+2j>0$ for $j\in\{0,1, 2\}$. Let $(\varphi_0,\varphi_1,\varphi_2)\in (H^{s+\ell+2}\cap L^1)\times (H^{s+\ell+1}\cap L^1)\times (H^{s+\ell}\cap L^1)$ with $\ell\geqslant 0$. Then, the Sobolev solution to the weakly damped MGT equation \eqref{Eq-Dissipation-MGT} in the critical case satisfies
	\begin{align*}
		\|\partial_t^j\varphi(t,\cdot)\|_{\dot{H}^{s+2-j}}\lesssim (1+t)^{-\min\left\{\frac{n+2s+4+2j}{4},\frac{\ell}{2}\right\}}\|(\varphi_0,\varphi_1,\varphi_2)\|_{(H^{s+\ell+\max\{1,2-j\}}\cap L^1)\times (H^{s+\ell+1}\cap L^1)\times (H^{s+\ell}\cap L^1)}.
	\end{align*}
	Additionally, by assuming the higher regularity of loss $\ell>\frac{n}{2}+s+2+j$ and $P_{\Phi_0}\neq0$, the optimal estimate
	\begin{align*}
	\|\partial_t^j\varphi(t,\cdot)\|_{\dot{H}^{s+2-j}}\approx t^{-\frac{n+2s+4+2j}{4}}
	\end{align*}
	holds for large time $t\gg1$. Furthermore, the refined estimate
	\begin{align*}
	&\big\|\partial_t^j\varphi(t,\cdot)-\chi_{\intt}(D)|D|^{2j}\ml{G}(t,|D|)\Phi_0(\cdot)\\
	&\qquad\quad\,\ \ -\chi_{\extt}(D)\,\mathrm{e}^{-\frac{\gamma}{2\tau^2}|D|^{-2}t}\partial_t^j\big[\ml{H}_{\mathrm{s}}(t,|D|)\Phi_1(\cdot)-\ml{H}_{\mathrm{c}}(t,|D|)\varphi_2(\cdot)\big]\big\|_{\dot{H}^{s+2-j}}\\
	&\lesssim (1+t)^{-\min\left\{\frac{n+2s+4+2j}{4}+1,\frac{\ell}{2}+\frac{1}{2}\right\}}\|(\varphi_0,\varphi_1,\varphi_2)\|_{(H^{s+\ell+\max\{1,2-j\}}\cap L^1)\times (H^{s+\ell+1}\cap L^1)\times (H^{s+\ell}\cap L^1)}
	\end{align*}
	holds for any $\ell\geqslant 0$.
\end{theorem}
\begin{proof}
The estimate for small frequencies is the same as the one in Theorem \ref{Theorem-Linear-Estimates} (so that $n+2s+4+2j>0$ for $j\in\{0,1, 2\}$ is needed), hence,
\begin{align*}
\|\partial_t^j\varphi(t,\cdot)\|_{\dot{H}^{s+2-j}}&\lesssim (1+t)^{-\frac{n+2s+4+2j}{4}}\|(\varphi_0,\varphi_1,\varphi_2)\|_{(L^1)^3}\\
&\quad+(1+t)^{-\frac{\ell}{2}}\|(\varphi_0,\varphi_1,\varphi_2)\|_{H^{s+\ell+\max\{1,2-j\}}\times H^{s+\ell+1}\times H^{s+\ell}},
\end{align*}
in which we used Lemma \ref{Lemma-3.2} for large frequencies combined with \eqref{es2.2.1}. Let us next conclude its optimality by setting the higher regularity of loss $2\ell>n+2s+4+2j$ so that
\begin{align*}
	\|\partial_t^j\varphi(t,\cdot)\|_{\dot{H}^{s+2-j}}\lesssim (1+t)^{-\frac{n+2s+4+2j}{4}}\|(\varphi_0,\varphi_1,\varphi_2)\|_{(H^{s+\ell+\max\{1,2-j\}}\cap L^1)\times (H^{s+\ell+1}\cap L^1)\times (H^{s+\ell}\cap L^1)}.
\end{align*}
Replacing \eqref{p3.2.1}, we arrive at another error estimate
\begin{align}
&\|\partial_t^j\varphi(t,\cdot)-|D|^{2j}\ml{G}(t,|D|)\Phi_0(\cdot)\|_{\dot{H}^{s+2-j}}\notag\\
&\lesssim t^{-\min\left\{\frac{n+2s+4+2j}{4}+1,\frac{\ell}{2}\right\}}\|(\varphi_0,\varphi_1,\varphi_2)\|_{(H^{s+\ell+\max\{1,2-j\}}\cap L^1)\times (H^{s+\ell+1}\cap L^1)\times (H^{s+\ell}\cap L^1)}\notag\\
&=o\big(t^{-\frac{n+2s+4+2j}{4}}\big)\label{Ineq-01}
\end{align}
for large time $t\gg1$, thanks to the strict restriction of $\ell$. Via \eqref{p3.2.3} as well as the triangle inequality, we immediately complete its sharp lower bound estimate. Finally, applying \eqref{es1.1.2} and \eqref{es2.2.2}, we get
\begin{align*}
&\big\|\partial_t^j\widehat{\varphi}(t,\xi)-\chi_{\intt}(\xi)|\xi|^{2j}\widehat{\ml{G}}(t,|\xi|)\widehat{\Phi}_0(\xi)-\chi_{\extt}(\xi)\,\mathrm{e}^{-\frac{\gamma}{2\tau^2}|\xi|^{-2}t}\partial_t^j\big[\widehat{\ml{H}}_{\mathrm{s}}(t,|\xi|)\widehat{\Phi}_1(\xi)-\widehat{\ml{H}}_{\mathrm{c}}(t,|\xi|)\widehat{\varphi}_2(\xi)\big]\big\|_{L^{2,s+2-j}}\\
&\lesssim (1+t)^{-\frac{n+2s+4+2j}{4}-1}\|(\varphi_0,\varphi_1,\varphi_2)\|_{(L^1)^3}+(1+t)^{-\frac{\ell+1}{2}}\|(\varphi_0,\varphi_1,\varphi_2)\|_{H^{s+\ell+\max\{1,2-j\}}\times H^{s+\ell+1}\times H^{s+\ell}}
\end{align*}
to end our proof.
\end{proof}

\begin{remark}
By subtracting the next function:
\begin{align*}
\chi_{\intt}(D)|D|^{2j}\ml{G}(t,|D|)\Phi_0(\cdot) +\chi_{\extt}(D)\,\mathrm{e}^{-\frac{\gamma}{2\tau^2}|D|^{-2}t}\partial_t^j\big[\ml{H}_{\mathrm{s}}(t,|D|)\Phi_1(\cdot)-\ml{H}_{\mathrm{c}}(t,|D|)\varphi_2(\cdot)\big]
\end{align*}
in the $L^2$ framework, the optimal decay rate of $\|\partial_t^j\varphi(t,\cdot)\|_{\dot{H}^{s+2-j}}$ has been improved at least $(1+t)^{-\frac{1}{2}}$. Consequently, it can be explained by the large time asymptotic profile. Note that $-\partial_t\ml{H}_{\mathrm{c}}(t,|D|)=\ml{H}_{\mathrm{s}}(t,|D|)$ which is the Fourier multiplier for the free wave equation. It shows the acoustic wave influence from this profile.
\end{remark}

\begin{remark}
	We are going to compare our result in Theorem \ref{Theorem-Linear-Estimates-2} with those for the following well-studied evolution models:
	\begin{itemize}
		\item the damped wave equation $\varphi^{\tau=0}_{tt}-\Delta \varphi^{\tau=0}+\gamma\varphi^{\tau=0}_t=0$, which is the singular limit case $\tau=0$ in \eqref{Eq-Dissipation-MGT} by setting $\delta=0$ and $\gamma>0$;
		\item the critical MGT equation $\tau\varphi^{\gamma=0}_{ttt}+\varphi^{\gamma=0}_{tt}-\Delta \varphi^{\gamma=0}-\tau\Delta \varphi^{\gamma=0}_t=0$, which is the vanishing weak damping limit case $\gamma=0$ in \eqref{Eq-Dissipation-MGT} by setting $\tau>0$ and $\delta=0$;
	\end{itemize}
	in the whole space $\mb{R}^n$ via the next table.
	\begin{table}[http]
		\centering	
		\begin{tabular}{cccc}
			\toprule
			Model  & Reference & Optimal rate  & Multiplier \\
			\midrule
			Weakly damped MGT equation& \multirow{2}{*}{Theorem \ref{Theorem-Linear-Estimates-2}} & $t^{-\min\left\{\frac{n}{4},\frac{\ell}{2}\right\}}=t^{-\frac{n}{4}}$ & \multirow{2}{*}{$\frac{1}{\gamma}\,\mathrm{e}^{-\frac{1}{\gamma}|\xi|^2t}$ and  $\mathrm{e}^{-\frac{\gamma}{2\tau^2}|\xi|^{-2}t}$}\\
			\eqref{Eq-Dissipation-MGT} with $\tau>0$, $\delta=0$, $\gamma>0$&&regularity-loss $\ell> \frac{n}{2}$&\\
			\midrule
			Damped wave equation& \multirow{2}{*}{\cite{Chill-Haraux=2003,Takeda=2015}} & \multirow{2}{*}{$t^{-\frac{n}{4}}$} & \multirow{2}{*}{$\frac{1}{\gamma}\,\mathrm{e}^{-\frac{1}{\gamma}|\xi|^2t}$}\\
			\eqref{Eq-Dissipation-MGT} with $\tau=0$, $\delta=0$, $\gamma>0$&&&\\
			\midrule
			Critical MGT equation  & \multirow{2}{*}{\cite{Chen-Palmieri=2020,Chen-Ikehata=2026,Takeda=2026}} & $\sqrt{t}$ if $n=1$; $\sqrt{\ln t}$ if $n=2$; & \multirow{2}{*}{$\frac{\sin(|\xi|t)}{|\xi|}$}\\
			\eqref{Eq-Dissipation-MGT} with $\tau>0$, $\delta=0$, $\gamma=0$&&$\approx C$ if $n\geqslant 3$&\\
			\bottomrule
			\multicolumn{4}{l}{$\star$ The optimal estimates and asymptotic profiles for large time are served for the solution itself.}
		\end{tabular}
		\caption{A comparison of sharp large time behavior in the critical case $\delta=0$}
		\label{Table_2}
	\end{table}
	
	\noindent We notice from Table \ref{Table_2} a new phenomenon that the weak damping term $+\gamma\varphi_t$ brings the regularity-loss decay property, which is caused by the Fourier multiplier $\ml{F}^{-1}_{\xi\to x}\left(\chi_{\extt}(\xi)\,\mathrm{e}^{-\frac{\gamma}{2\tau^2}|\xi|^{-2}t}\right)$ with $\tau>0$ and $\gamma>0$. The weak damping term eliminates the large time instability for $n\leqslant 2$ and the boundedness for $n\geqslant 3$ in the critical MGT equation. However, in the singular limit case $\tau=0$, the regularity-loss effect is dropped again. 
\end{remark}
 
\subsubsection{A comparison between the sub-critical and critical cases}	
\hspace{5mm}To end this part, let us address the similarity and difference between the sub-critical case $\delta>0$ and the critical case $\delta=0$ in the weakly damped MGT equation \eqref{Eq-Dissipation-MGT}. This comparison arises from Theorem \ref{Theorem-Linear-Estimates} and Theorem \ref{Theorem-Linear-Estimates-2}.
\begin{description}
	\item[Similarity:] They have the same optimal decay rate $(1+t)^{-\frac{n+2s+4+2j}{4}}$ if $P_{\Phi_0}\neq0$ due to the identical dominant multiplier $\frac{1}{\gamma}\,\mathrm{e}^{-\frac{1}{\gamma}|\xi|^2t}$ for small frequencies. This optimal decay rate heavily relies on the weak damping term $+\gamma\varphi_t$.
	\item[Difference:] The critical case requires the higher $\ell$-order Sobolev regularity with $\ell>\frac{n}{2}+s+2+j$ for all initial data to achieve the optimality due to its dominant multiplier $\mathrm{e}^{-\frac{\gamma}{2\tau^2}|\xi|^{-2}t}$ for large frequencies. This regularity-loss effect heavily relies on the viscous damping term $-\delta\Delta\varphi_t$.
\end{description}
Motivated by this comparison for the weakly damped MGT equation \eqref{Eq-Dissipation-MGT}, the solution $\varphi^{\delta>0}=\varphi^{\delta>0}(t,x)$ in the sub-critical case converges to the one $\varphi^{\delta=0}=\varphi^{\delta=0}(t,x)$ in the critical case for large time $t\gg1$ in the time-weighted Sobolev space. 

\begin{coro}\label{Coro-Convergence}
	Let $\tau>0$, $\delta\geqslant 0$, $0<\gamma\neq\frac{1}{4\tau}$, and $n+2s+4+2j>0$ for $j\in\{0,1, 2\}$. Let $(\varphi_0,\varphi_1,\varphi_2)\in (H^{s+\ell+2}\cap L^1)\times (H^{s+\ell+1}\cap L^1)\times (H^{s+\ell}\cap L^1)$ with $\ell>\frac{n}{2}+s+2+j$. Then, the Sobolev solutions to the weakly damped MGT equation \eqref{Eq-Dissipation-MGT} satisfy
	\begin{align*}
		\lim\limits_{t\to+\infty}t^{\frac{n+2s+4+2j}{4}}\|\partial_t^j\varphi^{\delta>0}(t,\cdot)-\partial_t^j\varphi^{\delta=0}(t,\cdot)\|_{\dot{H}^{s+2-j}}=0.
	\end{align*}
\end{coro}
\begin{proof}
Its justification is based on the intermediary $|D|^{2j}\ml{G}(t,|D|)\Phi_0(\cdot)$ in the triangle inequality associated with \eqref{p3.2.1} and \eqref{Ineq-01}. We omit its detail.
\end{proof}


\section{Weakly damped JMGT equation in the sub-critical case}\setcounter{equation}{0}\label{Section-Nonlinear}
\subsection{Preliminary results}\label{Sub-Section-Pre-Nonlinear}
\hspace{5mm}Strongly motivated by Proposition \ref{PROP-3.1}, let us introduce the following time-weighted Sobolev space for the solution:
\begin{align*}
	X_s(T):=\ml{C}([0,T],H^{s+2})\cap \ml{C}^1([0,T],H^{s+1})\cap \ml{C}^2([0,T],H^s)
\end{align*}
for any $T>0$, equipping its corresponding norm
\begin{align*}
	\|\psi\|_{X_s(T)}:=\sup\limits_{t\in[0,T]}\left(\,\sum\limits_{\sigma\in\{j-2,s\}}\sum\limits_{j\in\{0,1,2\}}(1+t)^{\frac{n+2\sigma+4+2j}{4}}\|\partial_t^j\psi(t,\cdot)\|_{\dot{H}^{\sigma+2-j}} \right),
\end{align*}
where we restricted $s>\max\{\frac{n}{2}-1,0\}$. The above time-dependent weights are the reciprocal functions of decay rates from the derived $(L^2\cap L^1)-L^2$ estimates in Theorem \ref{Theorem-Linear-Estimates} for its corresponding linearized Cauchy problem.

\subsubsection{Representations and fixed point argument}
\hspace{5mm}
According to Duhamel's principle, we as usual introduce the integral operator 
\begin{align}\label{A5}
	N:\ \psi\in X_s(T)\to N\psi:=\psi^{\lin}+\psi^{\non},
\end{align}
in which $\psi^{\lin}=\psi^{\lin}(t,x)$ solves the linearized model \eqref{Eq-Dissipation-MGT}, moreover, the function $\psi^{\non}=\psi^{\non}(t,x)$ is defined by
\begin{align*}
	\psi^{\non}(t,x):=\int_0^tK_2(t-\eta,x)\ast_{(x)}\partial_t\ml{N}(\psi_t,\nabla\psi)(\eta,x)\,\mathrm{d}\eta.
\end{align*}
Here, $K_2(t,x)$ denotes the kernel for the third (last) initial data in the linearized problem \eqref{Eq-Dissipation-MGT}, whose representation in the Fourier space is determined according to Section \ref{Sub-Section-Fourier-space}. Namely, this is the fundamental solution to the linear Cauchy problem \eqref{Eq-Dissipation-MGT} with the initial data $\varphi_0=\varphi_1=0$ and $\varphi_2=\delta_0$ with  the Dirac distribution $\delta_0$ at $x=0$. From the facts that
\begin{align*}
\widehat{K}_2(0,|\xi|)=\partial_t\widehat{K}_2(0,|\xi|)=0\ \ \mbox{and} \ \ \partial_t^2\widehat{K}_2(0,|\xi|)=1,
\end{align*}
 by using an integration by parts with respect to the time variable, we may arrive at  the crucial representation of $\partial_t^j\psi^{\non}$ via
\begin{align}
	\partial_t^j\psi^{\non}(t,x)&=-\partial_t^j K_2(t,x)\ast_{(x)}\ml{N}(\psi_1,\nabla\psi_0)(x)+\int_0^t\partial_t^{j+1}K_2(t-\eta,x)\ast_{(x)}\ml{N}(\psi_t,\nabla\psi)(\eta,x)\,\mathrm{d}\eta\label{Rep-non}
\end{align}
for $j\in\{0,1\}$.  The second-order derivative is reformulated via
\begin{align}
	\partial_t^2\psi^{\non}(t,x)&= \ml{N}(\psi_t,\nabla\psi)(t,x) - \partial_t^2 K_2(t,x)\ast_{(x)}\ml{N}(\psi_1,\nabla\psi_0)(x)\notag\\
	&\quad+\int_0^t\partial_t^{3}K_2(t-\eta,x)\ast_{(x)}\ml{N}(\psi_t,\nabla\psi)(\eta,x)\,\mathrm{d}\eta,\label{Cn} 
\end{align}
in which the additional nonlinear term $\ml{N}(\psi_t,\nabla\psi)$ appears different from $j\in\{0,1\}$. This new term actually will bring some technical difficulties when we study global in-time properties of solution if $n=1$, whose reason is postponed for elaboration. Soon afterwards, another formula of $\partial_t^2\psi^{\non}$ will be given in \eqref{Rep-non-2}.

We will demonstrate the global in-time existence and uniqueness of  small data Sobolev solution to the nonlinear problem \eqref{Eq-Dissipation-JMGT} by proving the existence of unique fixed point of $N$ in the space $X_s(+\infty)$. For this reason, we need to justify the following inequalities uniformly with respect to $T>0$ (meaning that all unexpressed multiplicative constants are independent of $T$):
\begin{align}
	\|N\psi\|_{X_s(T)}&\lesssim \|(\psi_0,\psi_1,\psi_2)\|_{\ml{A}_s}+\|\psi\|_{X_s(T)}^2,\label{Crucial-01}\\
	\|N\psi-N\bar{\psi}\|_{X_s(T)}&\lesssim \|\psi-\bar{\psi}\|_{X_s(T)}\big(\|\psi\|_{X_s(T)}+\|\bar{\psi}\|_{X_s(T)}\big),\label{Crucial-02}
\end{align}
for any $\psi,\bar{\psi}\in X_s(T)$, with the sufficiently small initial data in the topology of
\begin{align*}
\ml{A}_s:=(H^{s+2}\cap L^1)\times (H^{s+1}\cap L^1)\times (H^s\cap L^1).
\end{align*}
Therefore, \eqref{Crucial-01} together with \eqref{Crucial-02} derives the existence of unique Sobolev solution $\psi=N\psi\in X_s(T)$ by applying the Banach fixed point theorem. Furthermore, $\psi$ can be globally in-time prolonged because \eqref{Crucial-01} and \eqref{Crucial-02} hold uniformly with respect to $T$.

\subsubsection{Tools from the harmonic analysis}
\hspace{5mm}We in this part collect some useful inequalities from the harmonic analysis that will be used to estimate the nonlinear terms in homogeneous Sobolev spaces.

\begin{lemma}[Fractional Gagliardo-Nirenberg inequality, \cite{Hajaiej-Molinet-Ozawa-Wang-2011}]\label{fractionalgagliardonirenbergineq} 
		Let $p,p_0,p_1\in(1,+\infty)$ and $\kappa\in[0,s)$ with $s\in(0,+\infty)$. Then, the following fractional Gagliardo-Nirenberg inequality:
	\begin{align*}
		\|f\|_{\dot{H}^{\kappa}_{p}}\lesssim\|f\|_{L^{p_0}}^{1-\beta}\|f\|^{\beta}_{\dot{H}^{s}_{p_1}},
	\end{align*}
holds for $f\in L^{p_0}\cap \dot{H}^{s}_{p_1}$,	where  $\beta:=\frac{\frac{1}{p_0}-\frac{1}{p}+\frac{\kappa}{n}}{\frac{1}{p_0}-\frac{1}{p_1}+\frac{s}{n}}$ and $ \beta\in[\frac{\kappa}{s},1]$.
\end{lemma}

\begin{lemma}[Fractional Leibniz rule, \cite{Grafakos-Oh-2014}]\label{fractionleibnizrule}
	Let $s\in(0,+\infty)$, $r\in[1,+\infty]$ and $p_1,p_2,q_1,q_2\in(1,+\infty]$ satisfying the relation $\frac{1}{r}=\frac{1}{p_1}+\frac{1}{p_2}=\frac{1}{q_1}+\frac{1}{q_2}$. Then, the following fractional Leibniz rule:
\begin{align*}
	\|fg\|_{\dot{H}^{s}_{r}}\lesssim \|f\|_{\dot{H}^{s}_{p_1}}\|g\|_{L^{p_2}}+\|f\|_{L^{q_1}}\|g\|_{\dot{H}^{s}_{q_2}}
\end{align*}
holds for $f\in  \dot{H}^{s}_{p_1}\cap L^{q_1}$ and $g\in L^{p_2}\cap \dot{H}^{s}_{q_2}$.
\end{lemma}

\begin{lemma}[Fractional Sobolev embedding, \cite{Grafakos=2009}]\label{fractionembedd} Let $-\infty<\alpha_0<\frac{n}{2}<\beta_0<+\infty$. Then, the following fractional Sobolev embedding:
	\begin{equation*}
		\|f\|_{L^{\infty}}\lesssim\|f\|_{\dot{H}^{\alpha_0}}^{\frac{2\beta_0-n}{2(\beta_0-\alpha_0)}}\|f\|_{\dot{H}^{\beta_0}}^{\frac{n-2\alpha_0}{2(\beta_0-\alpha_0)}}
	\end{equation*}
	holds for $f\in \dot{H}^{\alpha_0}\cap \dot{H}^{\beta_0}$.
\end{lemma}

\subsection{A-priori estimates for the nonlinear parts}
\hspace{5mm}We firstly estimate the nonlinear initial data $\ml{N}(\psi_1,\nabla\psi_0)$ from \eqref{Rep-non} and \eqref{Cn} in the $L^2\cap L^1$ and $\dot{H}^s\cap L^1$ norms. Applying the fractional Gagliardo-Nirenberg inequality in Lemma \ref{fractionalgagliardonirenbergineq}, we get
\begin{align*}
	\|\ml{N}(\psi_1,\nabla\psi_0)\|_{L^2\cap L^1}&\lesssim \|\psi_0\|_{\dot{H}^1}^2+\|\psi_0\|_{\dot{H}^1_4}^2+\|\psi_1\|_{L^2}^2+\|\psi_1\|_{L^4}^2\\
	&\lesssim \|\psi_0\|_{\dot{H}^1}^2+\|\psi_0\|_{L^2}^{2-\frac{n+4}{2(s+2)}}\|\psi_0\|_{\dot{H}^{s+2}}^{\frac{n+4}{2(s+2)}}+\|\psi_1\|_{L^2}^2+\|\psi_1\|_{L^2}^{2-\frac{n}{2(s+1)}}\|\psi_1\|_{\dot{H}^{s+1}}^{\frac{n}{2(s+1)}}\\
	&\lesssim\|\psi_0\|_{H^{s+2}}^2+\|\psi_1\|_{H^{s+1}}^2
\end{align*}
and
\begin{align*}
	\|\ml{N}(\psi_1,\nabla\psi_0)\|_{\dot{H}^s\cap L^1}&\lesssim \|\nabla\psi_0\|_{L^{\infty}}\|\psi_0\|_{\dot{H}^{s+1}}+\|\psi_1\|_{L^{\infty}}\|\psi_1\|_{\dot{H}^s}+\|\psi_0\|_{\dot{H}^1}^2+\|\psi_1\|_{L^2}^2\\
	&\lesssim \|\psi_0\|_{L^2}^{\frac{2(s+1)-n}{2(s+2)}}\|\psi_0\|_{\dot{H}^{s+2}}^{\frac{n+2}{2(s+2)}}\|\psi_0\|_{\dot{H}^{s+1}}+\|\psi_1\|_{L^2}^{\frac{2(s+1)-n}{2(s+1)}}\|\psi_1\|_{\dot{H}^{s+1}}^{\frac{n}{2(s+1)}}\|\psi_1\|_{\dot{H}^s}\\
	&\quad +\|\psi_0\|_{\dot{H}^1}^2+\|\psi_1\|_{L^2}^2\\
	&\lesssim\|\psi_0\|_{H^{s+2}}^2+\|\psi_1\|_{H^{s+1}}^2
\end{align*}
by using Lemma \ref{fractionembedd} with $(\alpha_0,\beta_0)\in\{(-1,s+1),(0,s+1)\}$ thanks to $s+1>\frac{n}{2}$.

Obviously,
\begin{align*}
\|\,|\psi_t(\eta,\cdot)|^2\|_{L^1}&\lesssim (1+\eta)^{-\frac{n+4}{2}}\|\psi\|_{X_s(\eta)}^2,\\
	\|\,|\nabla\psi(\eta,\cdot)|^2\|_{L^1}&\lesssim (1+\eta)^{-\frac{n+2}{2}}\|\psi\|_{X_s(\eta)}^2,
\end{align*}
from the definition of evolution space $X_s(\eta)$.
Let us now estimate $|\psi_t(\eta,\cdot)|^2$ and $|\nabla\psi(\eta,\cdot)|^2$ in the $L^2$ and $\dot{H}^{s+1}$ norms, respectively. Actually, Lemma \ref{fractionalgagliardonirenbergineq} implies
\begin{align*}
	\|\,|\psi_t(\eta,\cdot)|^2\|_{L^2}\lesssim\|\psi_t(\eta,\cdot)\|_{L^4}^2&\lesssim\|\psi_t(\eta,\cdot)\|_{L^2}^{2-\frac{n}{2(s+1)}}\|\psi_t(\eta,\cdot)\|_{\dot{H}^{s+1}}^{\frac{n}{2(s+1)}}\\
	&\lesssim (1+\eta)^{-\frac{3n+8}{4}}\|\psi\|_{X_s(\eta)}^2,
\end{align*}
where we considered the condition $\frac{n}{4(s+1)}\in[0,1]$ which is always  true thanks to the assumption $s>\frac{n}{2}-1$. We make use of Lemma \ref{fractionembedd} to deduce
\begin{align*}
	\|\,|\psi_t(\eta,\cdot)|^2\|_{\dot{H}^{s+1}}&\lesssim\|\psi_t(\eta,\cdot)\|_{L^{\infty}}\|\psi_t(\eta,\cdot)\|_{\dot{H}^{s+1}}\\
	&\lesssim\|\psi_t(\eta,\cdot)\|_{L^2}^{1-\frac{n}{2(s+1)}}\|\psi_t(\eta,\cdot)\|_{\dot{H}^{s+1}}^{1+\frac{n}{2(s+1)}}\\
	&\lesssim(1+\eta)^{-\frac{3n+2s+10}{4}}\|\psi\|_{X_s(\eta)}^2.
\end{align*}
By using the same approach as the last lines, one easily finds
\begin{align*}
	\|\,|\nabla\psi(\eta,\cdot)|^2\|_{L^2}&\lesssim (1+\eta)^{-\frac{3n+4}{4}}\|\psi\|_{X_s(\eta)}^2,\\
	\|\,|\nabla\psi(\eta,\cdot)|^2\|_{\dot{H}^{s+1}}&\lesssim (1+\eta)^{-\frac{3n+2s+6}{4}}\|\psi\|_{X_s(\eta)}^2,
\end{align*}
due to $\frac{n+4}{4(s+2)}\in[\frac{1}{s+2},1]$.
The last three estimates for $|\nabla\psi(\eta,\cdot)|^2$ slower than the corresponding estimates for $|\psi_t(\eta,\cdot)|^2$ in the $L^1$, $L^2$ and $\dot{H}^{s+1}$ norms. That is to say,
\begin{align}
\|\ml{N}(\psi_t,\nabla\psi)(\eta,\cdot)\|_{\ml{X}\cap L^1}&\lesssim (1+\eta)^{-\frac{n+2}{2}}\|\psi\|_{X_s(\eta)}^2\label{D3}
\end{align}
for all $\ml{X}\in\{L^2,\dot{H}^{s+1}\}$, and
\begin{align} 
\|\ml{N}(\psi_t,\nabla\psi)(\eta,\cdot)\|_{L^2}&\lesssim (1+\eta)^{-\frac{3n+4}{4}}\|\psi\|_{X_s(\eta)}^2,\label{D4}\\
\|\ml{N}(\psi_t,\nabla\psi)(\eta,\cdot)\|_{\dot{H}^{s+1}}&\lesssim (1+\eta)^{-\frac{3n+2s+6}{4}}\|\psi\|_{X_s(\eta)}^2.\label{D5}
\end{align}

\subsection{Proof of Theorem \ref{Thm-GESDS}}
\subsubsection{Global in-time existence of Sobolev solution}
\hspace{5mm}Let us begin with the lower-order term $\partial_t^j N\psi$ for $j\in\{0,1\}$ introduced by \eqref{A5} together with  \eqref{Rep-non}. Theorem \ref{Theorem-Linear-Estimates} already derived the estimate for $\partial_t^j\psi^{\lin}$.
 In the next step, we apply the $(L^2\cap L^1)-L^2$ estimate in  $[0,\frac{t}{2}]$ and the $L^2-L^2$ estimate in $[\frac{t}{2},t]$ from Corollary \ref{CORO-3.1} (by taking $s=j-2$)
 \begin{align*}
 (1+t)^{\frac{n+4j}{4}}\|\partial_t^jN\psi(t,\cdot)\|_{L^2}&\lesssim \|(\psi_0,\psi_1,\psi_2)\|_{\ml{A}_s}+\|\ml{N}(\psi_1,\nabla\psi_0)\|_{L^2\cap L^1}\\
 &\quad+(1+t)^{\frac{n+4j}{4}}\int_0^{\frac{t}{2}}(1+t-\eta)^{-\frac{n+4+4j}{4}}\|\ml{N}(\psi_t,\nabla\psi)(\eta,\cdot)\|_{L^2\cap L^1}\,\mathrm{d}\eta\\
 &\quad+(1+t)^{\frac{n+4j}{4}}\int_{\frac{t}{2}}^t(1+t-\eta)^{-\frac{3+j}{2}}\|\ml{N}(\psi_t,\nabla\psi)(\eta,\cdot)\|_{L^2}\,\mathrm{d}\eta.
 \end{align*}
 Furthermore, with the help of estimates \eqref{D3} and \eqref{D4}, we arrive at
\begin{align}
	&(1+t)^{\frac{n+4j}{4}}\|\partial_t^jN\psi(t,\cdot)\|_{L^2}\notag\\
	&\lesssim\|(\psi_0,\psi_1,\psi_2)\|_{\ml{A}_s}+\|(\psi_0,\psi_1)\|_{H^{s+2}\times H^{s+1}}^2+(1+t)^{-1}\int_0^{\frac{t}{2}}(1+\eta)^{-\frac{n+2}{2}}\,\mathrm{d}\eta\,\|\psi\|_{X_s(T)}^2\notag\\
	&\quad+(1+t)^{-\frac{n+2-2j}{2}}\int_{\frac{t}{2}}^t(1+t-\eta)^{-\frac{3+j}{2}}\,\mathrm{d}\eta\, \|\psi\|_{X_s(T)}^2\notag\\
	&\lesssim\|(\psi_0,\psi_1,\psi_2)\|_{\ml{A}_s}+\big((1+t)^{-1}+(1+t)^{-\frac{n+2-2j}{2}}\big)\|\psi\|_{X_s(T)}^2\notag\\
	&\lesssim\|(\psi_0,\psi_1,\psi_2)\|_{\ml{A}_s}+\|\psi\|_{X_s(T)}^2,\label{Est-07}
\end{align}
where we used
\begin{align*}
\|(\psi_0,\psi_1,\psi_2)\|_{\ml{A}_s}+\|(\psi_0,\psi_1)\|_{H^{s+2}\times H^{s+1}}^2&\leqslant \|(\psi_0,\psi_1,\psi_2)\|_{\ml{A}_s}\big(1+\|(\psi_0,\psi_1,\psi_2)\|_{\ml{A}_s}\big)\\
&\lesssim \|(\psi_0,\psi_1,\psi_2)\|_{\ml{A}_s}
\end{align*}
since the small size of initial data in $\ml{A}_s$. Analogously, via \eqref{D3} and \eqref{D5}, the solution in $\dot{H}^{s+2-j}$ can be estimated as follows:
\begin{align}
	&(1+t)^{\frac{n+2s+4+2j}{4}}\|\partial_t^jN\psi(t,\cdot)\|_{\dot{H}^{s+2-j}}\notag\\
	&\lesssim\|(\psi_0,\psi_1,\psi_2)\|_{\ml{A}_s}+\|(\psi_0,\psi_1)\|_{H^{s+2}\times H^{s+1}}^2+(1+t)^{-1}\int_0^{\frac{t}{2}}(1+\eta)^{-\frac{n+2}{2}}\,\mathrm{d}\eta\,\|\psi\|_{X_s(T)}^2\notag\\
	&\quad+(1+t)^{-\frac{n+1-j}{2}}\int_{\frac{t}{2}}^t(1+t-\eta)^{-\frac{3+j}{2}}\,\mathrm{d}\eta\,\|\psi\|_{X_s(T)}^2\notag\\
	&\lesssim\|(\psi_0,\psi_1,\psi_2)\|_{\ml{A}_s}+\big((1+t)^{-1}+(1+t)^{-\frac{n+1-j}{2}}\big)\|\psi\|_{X_s(T)}^2\notag\\
	&\lesssim\|(\psi_0,\psi_1,\psi_2)\|_{\ml{A}_s}+\|\psi\|_{X_s(T)}^2.\label{Est-08}
\end{align}

Let us turn to the second-order time-derivative by concerning $\sigma\in\{0,s\}$. However, the new difficulty for $n=1$ comes from the representation \eqref{Cn} due to the a-priori estimates
\begin{align*}
\|\partial_t^2\psi(t,\cdot)\|_{\dot{H}^{\sigma}}&\lesssim (1+t)^{-\frac{n+2\sigma+8}{4}}\|\psi\|_{X_s(T)},\\
\|\ml{N}(\psi_t,\nabla\psi)(t,\cdot)\|_{\dot{H}^{\sigma}}&\lesssim (1+t)^{-\frac{3n+2\sigma+4}{4}}\|\psi\|_{X_s(T)}^2.
\end{align*}
Namely, the first term in the representation \eqref{Cn} plays a dominant role comparing with its left-hand side when $n=1$ such that $-\frac{3n+2\sigma+4}{4}>-\frac{n+2\sigma+8}{4}$ when $n=1$, which prevent us from proving a uniform estimate with respect to $t$.
\begin{remark}\label{Rem-W-model}
Concerning the nonlinearity of Westervelt-type model \eqref{Westervelt}, thanks to
\begin{align*}
\|\ml{N}(\psi_t,\psi_t)(t,\cdot)\|_{\dot{H}^{\sigma}}\lesssim (1+t)^{-\frac{3n+2\sigma+8}{4}}\|\psi\|_{X_s(T)}^2
\end{align*}
as well as $-\frac{3n+2\sigma+8}{4}<-\frac{n+2\sigma+8}{4}$ for any $n\geqslant 1$, the previous difficulty does not appear. Hence, one still may use the representation \eqref{Cn} in the treatment of $\partial_t^2N\psi$.
\end{remark}
\noindent To overcome this difficulty, we separate the integration $[0,t]$ of \eqref{Cn} into $[0,\frac{t}{2}]$ and $[\frac{t}{2},t]$. Furthermore, an integration by parts is used again in $[\frac{t}{2},t]$ to cancel the troublesomeness $\ml{N}(\psi_t,\nabla\psi)$ and get
\begin{align}\label{Rep-non-2}
\partial_t^2\psi^{\non}(t,x)&=- \partial_t^2 K_2(t,x)\ast_{(x)}\ml{N}(\psi_1,\nabla\psi_0)(x)+ \partial_t^2 K_2(\tfrac{t}{2},x)\ast_{(x)}\ml{N}(\psi_t,\nabla\psi)(\tfrac{t}{2},x)\notag\\
&\quad+\int_0^{\frac{t}{2}}\partial_t^{3}K_2(t-\eta,x)\ast_{(x)}\ml{N}(\psi_t,\nabla\psi)(\eta,x)\,\mathrm{d}\eta\notag\\
&\quad+\int_{\frac{t}{2}}^t\partial_t^2K_2(t-\eta,x)\ast_{(x)}\partial_t\ml{N}(\psi_t,\nabla\psi)(\eta,x)\,\mathrm{d}\eta.
\end{align}
Although now we still have the troublesomeness $\ml{N}(\psi_t,\nabla\psi)$ on the right-hand side, its kernel $\partial_t^2 K_2(\tfrac{t}{2},x)$ provides a faster decay rate.
As our preparation, we have to estimate the derivative-type nonlinearity by Lemma \ref{fractionembedd} such that
\begin{align*}
\|\partial_t\ml{N}(\psi_t,\nabla\psi)(\eta,\cdot)\|_{\dot{H}^{\sigma}}&\lesssim \|\psi_t(\eta,\cdot)\psi_{tt}(\eta,\cdot)\|_{\dot{H}^{\sigma}}+\|\nabla\psi(\eta,\cdot)\cdot\nabla\psi_t(\eta,\cdot)\|_{\dot{H}^{\sigma}}\\
&\lesssim\|\psi_t(\eta,\cdot)\|_{L^{\infty}}\|\psi_{tt}(\eta,\cdot)\|_{\dot{H}^{\sigma}}+\|\nabla\psi(\eta,\cdot)\|_{L^{\infty}}\|\nabla\psi_t(\eta,\cdot)\|_{\dot{H}^{\sigma}}\\
&\lesssim\|\psi_t(\eta,\cdot)\|_{L^2}^{\frac{2(s+1)-n}{2(s+1)}}\|\psi_t(\eta,\cdot)\|_{\dot{H}^{s+1}}^{\frac{n}{2(s+1)}}\|\psi_{tt}(\eta,\cdot)\|_{\dot{H}^{\sigma}}\\
&\quad+\|\psi(\eta,\cdot)\|_{L^2}^{\frac{2(s+1)-n}{2(s+2)}}\|\psi(\eta,\cdot)\|_{\dot{H}^{s+2}}^{\frac{n+2}{2(s+2)}}\|\psi_{t}(\eta,\cdot)\|_{\dot{H}^{\sigma+1}}\\
&\lesssim(1+\eta)^{-\frac{3n+2\sigma+8}{4}}\|\psi\|_{X_s(\eta)}^2.
\end{align*}
Recalling Corollary \ref{CORO-3.1} (by taking $j=1$ and $s+1=\sigma$ in $[\frac{t}{2},t]$) additionally, one has
\begin{align*}
&(1+t)^{\frac{n+2\sigma+8}{4}}\|\partial_t^2N\psi(t,\cdot)\|_{\dot{H}^{\sigma}}\\
&\lesssim\|(\psi_0,\psi_1,\psi_2)\|_{\ml{A}_s}+\|\ml{N}(\psi_1,\nabla\psi_0)\|_{\dot{H}^{\sigma}\cap L^1}+(1+t)^{\frac{n+2\sigma}{4}}\|\ml{N}(\psi_t,\nabla\psi)(\tfrac{t}{2},\cdot)\|_{\dot{H}^\sigma}\\
	&\quad+(1+t)^{\frac{n+2\sigma+8}{4}}\int_0^{\frac{t}{2}}(1+t-\eta)^{-\frac{n+2\sigma+12}{4}}(1+\eta)^{-\frac{n+2}{2}}\,\mathrm{d}\eta\,\|\psi\|_{X_s(T)}^2\\
	&\quad+(1+t)^{\frac{n+2\sigma+8}{4}}\int_{\frac{t}{2}}^t(1+t-\eta)^{-2}(1+\eta)^{-\frac{3n+2\sigma+8}{4}}\,\mathrm{d}\eta\,\|\psi\|_{X_s(T)}^2.
\end{align*} 
So, by using $(1+t-\eta)\approx (1+t)$ when $\eta\in[0,\frac{t}{2}]$ and $(1+\eta)\approx (1+t)$ when $\eta\in[\frac{t}{2},t]$, concerning $\sigma\in\{0,s\}$, we derive
\begin{align}
	&(1+t)^{\frac{n+2\sigma+8}{4}}\|\partial_t^2N\psi(t,\cdot)\|_{\dot{H}^{\sigma}}\notag\\
	&\lesssim \|(\psi_0,\psi_1,\psi_2)\|_{\ml{A}_s}+(1+t)^{-\frac{n+2}{2}}\|\psi\|_{X_s(T)}^2+(1+t)^{-1}\int_0^{\frac{t}{2}}(1+\eta)^{-\frac{n+2}{2}}\,\mathrm{d}\eta\,\|\psi\|_{X_s(T)}^2\notag\\
	&\quad+(1+t)^{-\frac{n}{2}}\int_{\frac{t}{2}}^t(1+t-\eta)^{-2}\,\mathrm{d}\eta\,\|\psi\|_{X_s(T)}^2\notag\\
	&\lesssim \|(\psi_0,\psi_1,\psi_2)\|_{\ml{A}_s}+\big((1+t)^{-1}+(1+t)^{-\frac{n}{2}}\big)\|\psi\|_{X_s(T)}^2\notag\\
	&\lesssim \|(\psi_0,\psi_1,\psi_2)\|_{\ml{A}_s}+\|\psi\|_{X_s(T)}^2.\label{C6}
\end{align}
Combining all previous estimates, we complete the proof of desired estimate \eqref{Crucial-01}. Analogously, by using the H\"older inequality to deal with the difference between $\ml{N}(\psi_t,\nabla\psi)$ and $\ml{N}(\bar{\psi}_t,\nabla\bar{\psi})$, moreover, applying Lemma \ref{fractionleibnizrule} additionally, we conclude \eqref{Crucial-02}.

Hence, we rigorously proved the global in-time existence of small data Sobolev solution in $ X_s(+\infty)$. As a byproduct, from the smallness of initial data, since the following estimate:
\begin{align}\label{N1}
	\|\psi\|_{X_s(T)}\lesssim\|(\psi_0,\psi_1,\psi_2)\|_{\ml{A}_s}
\end{align}
holds for any $T>0$, we immediately conclude the validity of sharp upper bound estimate \eqref{Upper-Nonlinear} for any $j\in\{0,1,2\}$.

\subsubsection{Large time profile of Sobolev solution}
\hspace{5mm}Given the global in-time solution $\psi\in X_s(+\infty)$ to the nonlinear Cauchy problem \eqref{Eq-Dissipation-JMGT}, via the re-expression of profile \eqref{An-Rep} and the mild solution's formula \eqref{Rep-non}, we establish 
\begin{align*}
	&\partial_t^j\psi(t,x)-(-\Delta)^j\widetilde{\psi}(t,x)\\
	&=\big(\partial_t^j\psi^{\lin}(t,x)-(-\Delta)^jG(t,x)P_{\Psi_0}\big)-\big(\partial_t^jK_2(t,|D|)\ml{N}(\psi_1,\nabla\psi_0)(x)-\tau(-\Delta)^jG(t,x)P_{\ml{N}(\psi_1,\nabla\psi_0)}\big)\\
	&\quad+\int_0^t\partial_t^{j+1}K_2(t-\eta,x)\ast_{(x)}\ml{N}(\psi_t,\nabla\psi)(\eta,x)\,\mathrm{d}\eta\\
	&=:I_1(t,x)-I_2(t,x)+I_3(t,x)
\end{align*}
for all $j\in\{0,1\}$. Let $\sigma\in\{j-2,s\}$ here. From the error estimate for the linear part in Theorem \ref{Theorem-Linear-Estimates}, we know that
\begin{align}\label{C5}
	\|I_1(t,\cdot)\|_{\dot{H}^{\sigma+2-j}}&=o\big(t^{-\frac{n+2\sigma+4+2j}{4}}\big)
\end{align}
as large time. Taking $\varphi_0=\varphi_1=0$ as well as $\varphi_2=\ml{N}(\psi_1,\nabla\psi_0)$ in \eqref{p3.2.1} and \eqref{p3.2.3}, one arrives at
\begin{align}\label{B2}
	\|I_2(t,\cdot)\|_{\dot{H}^{\sigma+2-j}}&\leqslant \big\|\big(\partial_t^jK_2(t,|D|)-\tau|D|^{2j}\ml{G}(t,|D|)\big)\ml{N}(\psi_1,\nabla\psi_0)(\cdot)\big\|_{\dot{H}^{\sigma+2-j}}\notag\\
	&\quad+\big\|\tau|D|^{2j}\ml{G}(t,|D|)\ml{N}(\psi_1,\nabla\psi_0)(\cdot)-\tau|D|^{2j}G(t,\cdot)P_{\ml{N}(\psi_1,\nabla\psi_0)}\big\|_{\dot{H}^{\sigma+2-j}}\notag\\
	&=o\big(t^{-\frac{n+2\sigma+4+2j}{4}}\big)
\end{align}
for large time $t\gg1$. Moreover, in the chain of inequalities that led to  \eqref{Est-07} and \eqref{Est-08}, we already obtained
\begin{align*}
	\|I_3(t,\cdot)\|_{\dot{H}^{\sigma+2-j}}&\lesssim\big((1+t)^{-1}+(1+t)^{-\frac{n+2-2j}{2}}+(1+t)^{-\frac{n+1-j}{2}}\big)(1+t)^{-\frac{n+2\sigma+4+2j}{4}}\|\psi\|_{X_s(T)}^2\\
	&\lesssim(1+t)^{-\frac{n+2\sigma+4+2j}{4}-\min\{1,\frac{n}{2}\}}\|(\psi_0,\psi_1,\psi_2)\|_{\ml{A}_s}^2,
\end{align*}
where we used  \eqref{N1}. Thus,
\begin{align*}
\|\partial_t^j\psi(t,\cdot)-(-\Delta)^j\widetilde{\psi}(t,\cdot)\|_{\dot{H}^{\sigma+2-j}}\leqslant\sum\limits_{l\in\{1,2,3\}}\|I_l(t,\cdot)\|_{\dot{H}^{\sigma+2-j}}=o\big(t^{-\frac{n+2\sigma+4+2j}{4}}\big),
\end{align*}
that is to say, we proved \eqref{Asymptotic-Nonlinear} for $j\in\{0,1\}$. Applying  this estimate together with the Minkowski inequality and Lemma \ref{Lemma-3.1}, we find
\begin{align}
	\|\partial_t^j\psi(t,\cdot)\|_{\dot{H}^{\sigma+2-j}}&\gtrsim\|G(t,\cdot)\|_{\dot{H}^{\sigma+2+j}}|M_{\tau,\gamma}|- \|\partial_t^j\psi(t,\cdot)-(-\Delta)^j\widetilde{\psi}(t,\cdot)\|_{\dot{H}^{\sigma+2-j}}\notag\\
	&\gtrsim t^{-\frac{n+2\sigma+4+2j}{4}}|M_{\tau,\gamma}|\label{C4}
\end{align}
for large time $t\gg1$ if $M_{\tau,\gamma}\neq0$, which is exactly \eqref{Lower-Nonlinear} for $j\in\{0,1\}$.

For the second-order time-derivative $j=2$, we can employ the same method that we just used for the lower-order terms for $j\in\{0,1\}$. Indeed, \eqref{C5} and \eqref{B2} are satisfied for $j=2$ as well. The derived estimate \eqref{C6} implies
\begin{align*}
&\left\| \partial_t^2 K_2(\tfrac{t}{2},\cdot)\ast_{(x)}\ml{N}(\psi_t,\nabla\psi)(\tfrac{t}{2},\cdot)\right\|_{\dot{H}^{\sigma}}+\left\|\int_0^{\frac{t}{2}}\partial_t^{3}K_2(t-\eta,\cdot)\ast_{(x)}\ml{N}(\psi_t,\nabla\psi)(\eta,\cdot)\,\mathrm{d}\eta\right\|_{\dot{H}^{\sigma}}\\
&\quad+\left\|\int_{\frac{t}{2}}^t\partial_t^2K_2(t-\eta,\cdot)\ast_{(x)}\partial_t\ml{N}(\psi_t,\nabla\psi)(\eta,\cdot)\,\mathrm{d}\eta \right\|_{\dot{H}^{\sigma}}\\
&\lesssim (1+t)^{-\frac{n+2\sigma+8}{4}-\min\{1,\frac{n}{2}\}}\|(\psi_0,\psi_1,\psi_2)\|_{\ml{A}_s}^2.
\end{align*}
According to the error terms constructed from \eqref{Rep-non-2} as follows:
\begin{align*}
	&\partial_t^2\psi(t,x)-(-\Delta)^2\widetilde{\psi}(t,x)\\
	&=\big(\partial_t^2\psi^{\lin}(t,x)-(-\Delta)^2G(t,x)P_{\Psi_0}\big)-\big(\partial_t^2K_2(t,|D|)\ml{N}(\psi_1,\nabla\psi_0)(x)-\tau(-\Delta)^2G(t,x)P_{\ml{N}(\psi_1,\nabla\psi_0)}\big)\\
	&\quad+ \partial_t^2 K_2(\tfrac{t}{2},x)\ast_{(x)}\ml{N}(\psi_t,\nabla\psi)(\tfrac{t}{2},x)+\int_0^{\frac{t}{2}}\partial_t^{3}K_2(t-\eta,x)\ast_{(x)}\ml{N}(\psi_t,\nabla\psi)(\eta,x)\,\mathrm{d}\eta\\
	&\quad+\int_{\frac{t}{2}}^t\partial_t^2K_2(t-\eta,x)\ast_{(x)}\partial_t\ml{N}(\psi_t,\nabla\psi)(\eta,x)\,\mathrm{d}\eta,
\end{align*}
we deduce \eqref{Asymptotic-Nonlinear} for $j=2$. Then, by the same way as \eqref{C4}, the lower bound estimate \eqref{Lower-Nonlinear} for $j=2$ is completed.

\section*{Acknowledgments}
Wenhui Chen is supported in part by the National Natural Science Foundation of China (grant No. 12301270), Guangdong Basic and Applied Basic Research Foundation (grant No. 2025A1515010240). The authors thank Ryo Ikehata (Hiroshima University) for some suggestions in the preparation of this manuscript.


\begin{thebibliography}{99}
\bibitem{Abramov-1999}
\newblock O.V. Abramov.
\newblock \emph{High-Intensity Ultrasonics: Theory and Industrial Applications}.
\newblock CRC Press, 2019.
\bibitem{B-L-2020}
\newblock M. Bongarti, S. Charoenphon, I. Lasiecka.
\newblock Vanishing relaxation time dynamics of the Jordan Moore-Gibson-Thompson equation arising in nonlinear acoustics.
\newblock \emph{J. Evol. Equ.} \textbf{21} (2021), no. 3, 3553--3584.
\bibitem{Chen=2025}
\newblock W. Chen.
\newblock Global in-time existence of solutions for the complex-valued Jordan-Moore-Gibson-Thompson equations of Westervelt-type under different conditions on initial data.
\newblock \emph{Preprint} (2025). arXiv:2507.08273
\bibitem{Chen-Gong=2024}
\newblock W. Chen, J. Gong.
\newblock Some asymptotic profiles for the viscous Moore-Gibson-Thompson equation in the $L^q$ framework.
\newblock \emph{J. Math. Anal. Appl.} \textbf{540} (2024), no. 2, Paper No. 128641, 26 pp.
\bibitem{Chen-Ikehata=2021}
\newblock W. Chen, R. Ikehata.
\newblock The Cauchy problem for the Moore-Gibson-Thompson equation in the dissipative case.
\newblock \emph{J. Differential Equations} \textbf{292} (2021), 176--219.
\bibitem{Chen-Ikehata=2026}
\newblock W. Chen, R. Ikehata.
\newblock Large time behavior for the classical wave equation with different regular data and its applications.
\newblock \emph{Asymptot. Anal.} (2026), accepted.
\bibitem{Chen-Liu-Palmieri-Qin=2023}
\newblock W. Chen, Y. Liu, A. Palmieri, X. Qin.
\newblock The influence of viscous dissipations on the nonlinear acoustic wave equation with second sound.
\newblock \emph{Preprint} (2023). arXiv:2211.00944
\bibitem{Chen-Ma-Qin=2025}
\newblock W. Chen, M. Ma, X. Qin.
\newblock $L^p-L^q$ estimates for the dissipative and conservative Moore-Gibson-Thompson equations.
\newblock \emph{J. Math. Phys.} \textbf{66} (2025), no. 7, Paper No. 071515, 17 pp.
\bibitem{Chen-Palmieri=2020}
\newblock W. Chen, A. Palmieri. 
\newblock Nonexistence of global solutions for the semilinear Moore-Gibson-Thompson equation in the conservative case.
\newblock \emph{Discrete Contin. Dyn. Syst.} \textbf{40} (2020), no. 9, 5513--5540.
\bibitem{Chen-Palmieri=2021}
\newblock W. Chen, A. Palmieri.
\newblock A blow-up result for the semilinear Moore-Gibson-Thompson equation with nonlinearity of derivative type in the conservative case.
\newblock \emph{Evol. Equ. Control Theory} \textbf{10} (2021), no. 4, 673--687.
\bibitem{Chen-Takeda=2023}
\newblock W. Chen, H. Takeda. 
\newblock Asymptotic behaviors for the Jordan-Moore-Gibson-Thompson equation in the viscous case.
\newblock \emph{Nonlinear Anal.} \textbf{234} (2023), Paper No. 113316, 36 pp.
\bibitem{Chill-Haraux=2003}
\newblock R. Chill, A. Haraux.
\newblock An optimal estimate for the difference of solutions of two abstract evolution equations.
\newblock \emph{J. Differential Equations} \textbf{193} (2003), no. 2, 385--395.
\bibitem{Conejero-Lizama-Rodenas-2015}
\newblock J.A. Conejero, C. Lizama, F. Rodenas.
\newblock Chaotic behaviour of the solutions of the Moore-Gibson-Thompson equation.
\newblock \emph{Appl. Math. Inf. Sci.} \textbf{9} (2015), no. 5, 2233--2238.
\bibitem{Dao-Van-Nguyen=2024}
\newblock T.A. Dao, D. Van Duong, D.A. Nguyen.
\newblock On asymptotic properties of solutions to $\sigma$-evolution equations with general double damping.
\newblock \emph{J. Math. Anal. Appl.} \textbf{536} (2024), no. 2, Paper No. 128246, 35 pp.
\bibitem{Dell-Pata=2017}
\newblock F. Dell'Oro, V. Pata.
\newblock On the Moore-Gibson-Thompson equation and its relation to linear viscoelasticity.
\newblock \emph{Appl. Math. Optim.} \textbf{76} (2017), no. 3, 641--655.
\bibitem{Dreyer-Krauss-Bauer-Ried-2000}
\newblock T. Dreyer, W. Krauss, E. Bauer, R.E. Riedlinger.
\newblock Investigations of compact self focusing transducers using stacked piezoelectric elements for strong sound pulses in therapy.
\newblock \emph{IEEE Ultrasonics Sympos.}  \textbf{2} (2000), 1239--1242.
\bibitem{Duan=2011}
\newblock R. Duan.
\newblock Global smooth flows for the compressible Euler-Maxwell system. The relaxation case.
\newblock \emph{J. Hyperbolic Differ. Equ.} \textbf{8} (2011), no. 2, 375--413.
\bibitem{Grafakos=2009}
\newblock L. Grafakos.
\newblock \emph{Modern Fourier Analysis}.
\newblock Grad. Texts in Math., 250. Springer, New York, 2009.
\bibitem{Grafakos-Oh-2014} 
\newblock L. Grafakos, S. Oh.
\newblock The Kato-Ponce inequality.
\newblock \emph{Comm. Partial Differential Equations} \textbf{39} (2014), no. 6, 1128--1157.
\bibitem{Hajaiej-Molinet-Ozawa-Wang-2011}
\newblock H. Hajaiej, L. Molinet, T. Ozawa, B. Wang.
\newblock Necessary and sufficient conditions for the fractional Gagliardo-Nirenberg inequalities and applications to Navier-Stokes and generalized boson equations.
\newblock Harmonic analysis and nonlinear partial differential equations, 159--175,
\emph{RIMS K\^oky\^uroku Bessatsu, B26}, Research Institute for Mathematical Sciences (RIMS), Kyoto, 2011.
\bibitem{Hosono-Kawashima=2006}
\newblock T. Hosono, S. Kawashima.
\newblock Decay property of regularity-loss type and application to some nonlinear hyperbolic-elliptic system.
\newblock \emph{Math. Models Methods Appl. Sci.} \textbf{16} (2006), no. 11, 1839--1859.
\bibitem{Ide-Haramoto-Kawashima=2008}
\newblock K. Ide, K. Haramoto, S. Kawashima.
\newblock Decay property of regularity-loss type for dissipative Timoshenko system.
\newblock \emph{Math. Models Methods Appl. Sci.} \textbf{18} (2008), no. 5, 647--667.
\bibitem{Ikehata-Michihisa=2019}
\newblock R. Ikehata, H. Michihisa.
\newblock Moment conditions and lower bounds in expanding solutions of wave equations with double damping terms.
\newblock \emph{Asymptot. Anal.} \textbf{114} (2019), no. 1-2, 19--36.
\bibitem{Ikehata-Sawada=2016}
\newblock R. Ikehata, A. Sawada.
\newblock Asymptotic profile of solutions for wave equations with frictional and viscoelastic damping terms.
\newblock \emph{Asymptot. Anal.} \textbf{98} (2016), no. 1-2, 59--77.
\bibitem{Ikehata-Takeda=2017}
\newblock R. Ikehata, H. Takeda.
\newblock Critical exponent for nonlinear wave equations with frictional and viscoelastic damping terms.
\newblock \emph{Nonlinear Anal.} \textbf{148} (2017), 228--253.
\bibitem{Jordan-2014}
\newblock P.M. Jordan.
\newblock Second-sound phenomena in inviscid, thermally relaxing gases.
\newblock \emph{Discrete Contin. Dyn. Syst. Ser. B} \textbf{19} (2014), no. 7, 2189--2205.
\bibitem{Kaltenbacher-2025}
\newblock B. Kaltenbacher.
\newblock Acoustic nonlinearity parameter tomography with the Jordan-Moore-Gibson-Thompson equation in frequency domain.
\newblock \emph{Inverse Problems} \textbf{41} (2025), no. 9, Paper No. 095010, 27 pp.
\bibitem{Kaltenbacher-Lasiecka-Marchand-2011}
\newblock B. Kaltenbacher, I. Lasiecka, R. Marchand.
\newblock Wellposedness and exponential decay rates for the Moore-Gibson-Thompson equation arising in high intensity ultrasound.
\newblock \emph{Control Cybernet.} \textbf{40} (2011), no. 4, 971--988.
\bibitem{Kaltenbacher-Lasiecka-Pos-2012}
\newblock B. Kaltenbacher, I. Lasiecka, M.K. Pospieszalska. 
\newblock Well-posedness and exponential decay of the energy in the nonlinear Jordan-Moore-Gibson-Thompson equation arising in high intensity ultrasound.
\newblock \emph{Math. Models Methods Appl. Sci.} \textbf{22} (2012), no. 11, 1250035, 34 pp.
\bibitem{Kaltenbacher-Nikolic-2019}
\newblock B. Kaltenbacher, V. Nikoli\'c.
\newblock The Jordan-Moore-Gibson-Thompson equation: well-posedness with quadratic gradient nonlinearity and singular limit for vanishing relaxation time.
\newblock \emph{Math. Models Methods Appl. Sci.} \textbf{29} (2019), no. 13, 2523--2556.
\bibitem{Kaltenbacher-Niko-2021}
\newblock B. Kaltenbacher, V. Nikoli\'c.
\newblock The inviscid limit of third-order linear and nonlinear acoustic equations.
\newblock \emph{SIAM J. Appl. Math.} \textbf{81} (2021), no. 4, 1461--1482.
\bibitem{Kaltenbacher-Landes-Hoffelner-Simkovics-2002}
\newblock M. Kaltenbacher, H. Landes, J. Hoffelner, R. Simkovics.
\newblock Use of modern simulation for industrial applications of high power ultrasonics.
\newblock \emph{IEEE Ultrasonics Sympos.} \textbf{1} (2002), 673--678. 
\bibitem{Lighthill=1956}
\newblock M.J. Lighthill.
\newblock Viscosity effects in sound waves of finite amplitude. 
\newblock \emph{Surveys in mechanics} (1956), 250351.
\bibitem{Marchand-McDevitt-Triggiani-2012}
\newblock R. Marchand, T. McDevitt, R. Triggiani.
\newblock An abstract semigroup approach to the third-order Moore-Gibson-Thompson partial differential equation arising in high-intensity ultrasound: structural decomposition, spectral analysis, exponential stability.
\newblock \emph{Math. Methods Appl. Sci.} \textbf{35} (2012), no. 15, 1896--1929.
\bibitem{MooreGibson1960} 
\newblock F.K. Moore, W.E. Gibson.
\newblock Propagation of weak disturbances in a gas subject to relaxation effect.
\newblock \emph{J. Aerospace Sci.} \textbf{27} (1960), no. 2, 117--127.
\bibitem{Niko-Winker=2024}
\newblock V. Nikoli\'c, M. Winkler.
\newblock $L^{\infty}$ blow-up in the Jordan-Moore-Gibson-Thompson equation.
\newblock \emph{Nonlinear Anal.} \textbf{247} (2024), Paper No. 113600, 23 pp.
\bibitem{Pellicer-Said-Houari=2019}
\newblock M. Pellicer, B. Said-Houari. 
\newblock Wellposedness and decay rates for the Cauchy problem of the Moore-Gibson-Thompson equation arising in high intensity ultrasound.
\newblock \emph{Appl. Math. Optim.} \textbf{80} (2019), no. 2, 447--478.
\bibitem{Racke-Said-2020}
\newblock R. Racke, B. Said-Houari.
\newblock Global well-posedness of the Cauchy problem for the 3D Jordan-Moore-Gibson-Thompson equation.
\newblock \emph{Commun. Contemp. Math.} \textbf{23} (2021), no. 7, Paper No. 2050069, 39 pp.
\bibitem{Said-Houari=Besov=2022}
\newblock B. Said-Houari.
\newblock Global existence for the Jordan-Moore-Gibson-Thompson equation in Besov spaces.
\newblock \emph{J. Evol. Equ.} \textbf{22} (2022), no. 2, Paper No. 32, 40 pp.
\bibitem{Said-Houari=Large-Norm=2022}
\newblock B. Said-Houari.
\newblock Global well-posedness of the Cauchy problem for the Jordan-Moore-Gibson-Thompson equation with arbitrarily large higher-order Sobolev norms.
\newblock \emph{Discrete Contin. Dyn. Syst.} \textbf{42} (2022), no. 9, 4615--4635.
\bibitem{Takeda=2015}
\newblock H. Takeda.
\newblock Higher-order expansion of solutions for a damped wave equation.
\newblock \emph{Asymptot. Anal.} \textbf{94} (2015), no. 1-2, 1--31.
\bibitem{Takeda=2026}
\newblock H. Takeda.
\newblock $L^2$-Estimates for the linear elastic waves.
\newblock \emph{Math. Ann.} \textbf{394} (2026), no. 4, 82.
\bibitem{Thompson1972} 
\newblock P.A. Thompson.
\newblock \emph{Compressible-Fluid Dynamics}.
\newblock McGraw-Hill, New York, 1972.
\end{thebibliography}
\end{document}